\documentclass[11pt,usenames,dvipsnames]{article}
\usepackage[T1]{fontenc}
\usepackage{lmodern}
\usepackage{dsfont}
\usepackage{multicol}
\usepackage{enumerate}
\usepackage{amssymb,amsmath,amsthm,epsfig,textcomp,amsthm}
\usepackage[colorlinks=true,breaklinks=true,linkcolor=MidnightBlue!100,citecolor=Maroon!100]{hyperref}
\usepackage{graphics,psfrag,graphicx,color}
\usepackage[graphicx]{realboxes}
\numberwithin{equation}{section}
\usepackage[normalem]{ulem}

\usepackage{enumitem}  
\setlist{  
  listparindent=\parindent,
  parsep=0pt,
}

\usepackage[x11names,table]{xcolor}
\usepackage{booktabs}
\usepackage{rotating}
\usepackage{varwidth}
\usepackage{comment}
\usepackage{float}
\floatstyle{boxed} 

\usepackage{tabularx,ragged2e,booktabs,caption}
\newcolumntype{C}[1]{>{\Centering}m{#1}}

\usepackage{multirow}

\usepackage{algorithm}
\usepackage{algpseudocode}
\usepackage{etoolbox}
\makeatletter
\newcommand*{\algrule}[1][\algorithmicindent]{%
  \makebox[#1][l]{%
    \hspace*{.2em}
    \vrule height .75\baselineskip depth .25\baselineskip
  }
}

\newcount\ALG@printindent@tempcnta
\def\ALG@printindent{%
    \ifnum \theALG@nested>0
    \ifx\ALG@text\ALG@x@notext
    \else
    \unskip
    \ALG@printindent@tempcnta=1
    \loop
    \algrule[\csname ALG@ind@\the\ALG@printindent@tempcnta\endcsname]%
    \advance \ALG@printindent@tempcnta 1
    \ifnum \ALG@printindent@tempcnta<\numexpr\theALG@nested+1\relax
    \repeat
    \fi
    \fi
}
\patchcmd{\ALG@doentity}{\noindent\hskip\ALG@tlm}{\ALG@printindent}{}{\errmessage{failed to patch}}
\patchcmd{\ALG@doentity}{\item[]\nointerlineskip}{}{}{} 

\usepackage{mathrsfs}
\usepackage{mathtools}
\usepackage{authblk}

\newcommand{\pdual}[1]{\left\langle#1\right\rangle}

\newcommand{\jump}[1]{ [ \! [ {#1} ] \! ] }
\newcommand{\wtil}[1]{\widetilde{#1}}
\newcommand{\what}[1]{\widehat{#1}}
\newcommand{\kap}{\kappa}
\newcommand{\g}{\overline{\mathrm{g}}}
\newcommand{\compD}{\Omega_h}
\newcommand{\erroru}{\varepsilon^u}
\newcommand{\errorq}{\boldsymbol{\varepsilon}^{\boldsymbol{q}}}
\newcommand{\errortu}{\varepsilon^{\widehat{u}}}
\newcommand{\errortq}{\boldsymbol{\varepsilon}^{\widehat{\boldsymbol{q}}}}

\newcommand{\projerrorq}{\Lambda_{\boldsymbol{q}}}
\newcommand{\projerroru}{\Lambda_{u}}
\newcommand{\varphiJ}{\varphi_{\boldsymbol{q}}}
\newcommand{\aposteerror}{\textrm{e}_h}
\newcommand{\triple}[1]{|\!|\!|{#1}|\!|\!|}

\newcommand{\mb}[1]{\mathbb{#1}}
\newcommand{\mc}[1]{\mathcal{#1}}

\newcommand{\md}[1]{\mathds{#1}}
\newcommand{\mbf}[1]{\boldsymbol{#1}}
\newcommand{\bsy}[1]{\boldsymbol{#1}}
\newtheorem{thm}{Theorem}
\newtheorem{lem}{Lemma}
\newtheorem{crl}{Corollary}

\newcommand{\cO}{\mathcal{O}}
\newcommand\rmH{\mathrm{H}}
\newcommand{\bL}{\boldsymbol{L}}
\newcommand\bH{\boldsymbol{H}}

\title{{\it A priori} and {\it a posteriori} error analysis of an unfitted HDG method for semi-linear elliptic problems}

\author[1,2]{Nestor S\'anchez }
\author[3,4]{Tonatiuh S\'anchez-Vizuet}
\author[1,2]{Manuel E. Solano}
\affil[1]{{\footnotesize Departamento de Ingenier\'ia Matem\'atica, Facultad de Ciencias F\'isicas y Matem\'aticas, Universidad de Concepci\'on, Concepci\'on, Chile.}}
\affil[2]{{\footnotesize Centro de Investigaci\'on 
en Ingenier\'ia Matem\'atica (CI$^2$MA), Universidad de Concepci\'on, Concepci\'on, Chile.}}
\affil[3]{{\footnotesize New York University, Courant Institute of Mathematical Sciences, New York, USA.}}
\affil[4]{{\footnotesize Department of Mathematics, The University of Arizona. Tucson, AZ, USA.}}

\begin{document}
\date{}
\maketitle

\begin{abstract}
We present {\it a priori} and {\it a posteriori} error analysis of a high order hybridizable discontinuous Galerkin (HDG) method applied to a semi-linear elliptic problem posed on a piecewise curved, non polygonal domain. We approximate $\Omega$ by a polygonal subdomain $\Omega_h$ and propose an HDG discretization, which is shown to be optimal under mild assumptions related to the non-linear source term and the distance between the boundaries of the polygonal subdomain $\Omega_h$ and the true domain $\Omega$. Moreover, a local non-linear post-processing of the scalar unknown is proposed and shown to provide an additional order of convergence. A reliable and locally efficient {\it a posteriori} error estimator that takes into account the error in the approximation of the boundary data of $\Omega_h$ is also provided.
\end{abstract}


\section{Introduction}
%
In this paper, we carry out {\it a priori} and  {\it a posteriori} error analyses of a  hybridizable discontinuous Galerkin  (HDG) method \cite{Cockburn2010} applied to semi-linear elliptic problems of the form
	\begin{subequations}\label{eq:main}
	\begin{align}
	&& -\nabla \cdot \left( \kap \nabla u \right) &= \mc{F}(u) & &\text{ in } \Omega, && \\
	&& u &= g  & &\text{ on } \Gamma:=\partial \Omega, &&
	\end{align}
	\end{subequations}
where the domain $\Omega \subset \mathbb{R}^d$ ($d=2,3$) is not necessarily polygonal/polyhedral,  $\kap$ is a positive function in $\Omega$, $\mc{F}$ is a source term that depends on the solution $u$ and $g$ is the Dirichlet boundary data on $\Gamma$. To avoid the trivial solution, we will assume that if the boundary conditions are homogeneous, the source term will not vanish for $u=0$.

The authors's original motivation to study this type of problems comes from an application to plasma physics, where the magnetic equilibrium in axisymmetric fusion reactors can be described in terms of the solution of an equation of this type, known in the literature as the the \textit{Grad-Shafranov} equation \cite{GrRu1958,Shafranov1958}. Due to the symmetry of the device, the equation is posed in a two-dimensional domain, corresponding to a cross section of the reactor at a constant toroidal angle. The plasma confinement region is the domain enclosed by the zero level set of the solution, which is a piecewise smooth curve that in theoretical studies is often considered given and does not contain the vertical axis \cite{SaSo2018}. In fact, the Grad-Shafranov equation is nothing but \eqref{eq:main} where $g=0$, $\kap(x,y) = 1/x$ and the source term is a case-dependent function related to the current density in the toroidal direction and the pressure profile in the plasma. Note that in plasma applications, the domain $\Omega$ does not include the $y$ axis.

One of the first analysis of Galerkin methods for semi-linear elliptic problems can be traced back to the 80's where optimal $L^2$ error estimates were provided \cite{HaLa1987}. In the late nineties efficient methods in the spirit of multigrid methods were developed \cite{Xu1994,Xu1996}.  Since then, a vast literature can be found specially for parabolic problems, most of them focused on improving the computational efficiency of the numerical schemes, due to the fact that the non-linear source term requires the implementation of iterative schemes that might involve the computation of the Jacobian in the case of Newton's method. To circumvent this drawback, during the last decade, new methods have been proposed. For instance, adaptive iterative schemes \cite{Am2019,AmWi2015,PaWi2020,
HoWi2020} and interpolatory methods \cite{CoSiZh,XiCh2005}. In the context of discontinuos Galerkin method, we can mention  \cite{CoSiZh} and \cite{ZhZhJi}, whereas recently a virtual finite element method has been analyzed \cite{AsNaNa2019}. We emphasize that all aforementioned references consider exclusively polygonal/polyhedral domains. To the best of our knowledge, there are no works on {\it a posteriori} error analysis for semilinear problems posed in domains with piece-wise smooth boundary. We consider precisely this to be the main contribution of our work.

In the present study, the source term $\mc{F}$ will be assumed to be Lipschitz-continuous in $\Omega$, i.e, there exists  $L_{\Omega}>0$ such that
	\begin{equation}\label{eq:Lipschitz}
	\|\mc{F}(u_1) - \mc{F}(u_2)\|_{\Omega} \leq L_{\Omega}  \|u_1-u_2\|_{\Omega} \qquad \forall \, u_1,u_2 \in L^2(\Omega).
	\end{equation}
In addition, we assume that there exist positive constants $\underline{\kap}$ and $\overline{\kap}$ such that
\begin{eqnarray*}
    \underline{\kap} \leq \kap(\boldsymbol{x}) \leq\overline{\kap} \quad \forall \, \boldsymbol{x} \in \Omega.
    \end{eqnarray*}	
An HDG discretization requires us to formulate the problem in mixed from through the introduction of the flux $\mbf{q}:= -\kap\nabla u$ as an additional unknown. This choice makes it possible to write \eqref{eq:main} as the equivalent first order system 
	\begin{subequations}\label{eq:mixed}
	\begin{align}
	&& && \mbf{q} + \kap \nabla u &= 0 & &\text{ in } \Omega, &&\label{eq:mixed_a}\\
	&& && \nabla \cdot \mbf{q}  &= \mc{F}(u) & &\text{ in }  \Omega, &&\label{eq:mixed_b}\\
	&& && u &= g &  &\text{ on } \partial \Omega .&&\label{eq:mixed_c}
	\end{align}
	\end{subequations}
HDG schemes, as many other discretization methods, are based on a triangulation of the domain. In our case, $\Omega$ has a piecewise curved boundary which complicates the use high order methods, since the boundary must be properly interpolated by ``curved'' triangles or tetrahedra in order to preserve high order convergence. An alternative is to approximate $\Omega$ by a polygonal/polyhedral subdomain $\compD\subset\Omega$, that can be easily discretized by a uniform triangulation of size $h>0$. Then, the system \eqref{eq:mixed} can be restricted to $\compD$:
	\begin{subequations}\label{eq: Grad-Shaf. - comput domain}
	\begin{align}
	&& && \mbf{q} +\kappa \nabla u &= 0 & &\text{ in } \compD, \label{eq: GS - comput domain_a}\\
	&& && \nabla \cdot \mbf{q}  &= \mc{F}(u) & &\text{ in }   \compD, \label{eq: GS - comput domain_b} \\
	&& && u &= \varphi & &\text{ on } \Gamma_h:= \partial  \compD, \label{eq: GS - comput domain_c}
	\end{align}
	\end{subequations}
where the unknown $\varphi$ is the Dirichlet  data on the computational boundary  $\Gamma_h$. A clever way to determine $\varphi$ was proposed for one dimension in \cite{CoReGu2009} and then extended to higher dimensions by \cite{CoSo2012}. The method consists of using the definition of the flux to transfer the Dirichlet data from $\Gamma$ to $\Gamma_h$ along segments called \textit{transferring paths}. In fact, given $\boldsymbol{x}\in \Gamma_h$ and $\overline{\boldsymbol{x}} \in \Gamma$, one can integrate \eqref{eq:mixed_a} along a segment of length $l(\mbf{x})$ with unit tangent vector $\boldsymbol{t}(\mbf{x})$ connecting them to obtain the following representation for $\varphi$:
	\begin{equation}\label{def: varphi}
	\varphi(\mbf{x})= g(\overline{\mbf{x}})+ \int_0^{l(\mbf{x})} (\kap^{-1} \,\mbf{q})(\mbf{x} + \boldsymbol{t}(\mbf{x})s) \cdot  \boldsymbol{t}(\mbf{x}) ds.
	\end{equation}
 Above, we have considered that $u(\overline{\mbf{x}}) = g(\overline{\mbf{x}})$. At the end of Section \ref{sec:Notation} we will describe a way to pick $\overline{\boldsymbol{x}}$ in such a way that the transfer will preserve the order of approximation of the underlying discretization. Notice that the assumption \eqref{eq:Lipschitz} implies that $\mc{F}$ is also Lipschitz continuous in $\compD$ with constant $L\leq L_\Omega$; this observation will be useful in the analysis to follow.

In previous works the authors had applied this transfer technique in combination with an iterative HDG discretization to deal with the nonlinear system \eqref{eq: Grad-Shaf. - comput domain} arising from the Grad-Shafranov equation \cite{SaSo2018} and explored an h-adaptive HDG scheme for the solution of the problem \cite{SaCeSo2019}. The adaptive strategy was powered by a residual-based error estimator first proposed by Cockburn and Zhang \cite{aposteriori1}, albeit for polygonal domains---therefore not requiring the transfer of the boundary data--- and linear problems. The goal of this work is to provide a rigorous justification for the numerical results obtained previously by the authors when applying these techniques for semi-linear problems in curved geometries. The present communication is mainly theoretical and we refer the reader interested on numerical experiments to \cite{SaSo2018,SaCeSo2019} where plenty of experiments are provided within the context of plasma equilibrium. The results presented here, however, are not limited to plasma applications and remain valid for general semi-linear elliptic equations .
\section{The discrete scheme}\label{sec:DiscreteScheme}
%
\subsection{Basic  Notation.}\label{sec:Notation}
%
\paragraph{The Computational Domain.}\label{sec:ComputationalDomain}
Let $\{\mc{T}_{h}\}_{h>0}$ be a family of simplicial triangulations of $\Omega_h$ that will be assumed to be shape-regular. i.e., there exists $\beta>0$ such that for all elements $T\in \mc{T}_h$ and all $h>0$, $h_T / \rho_T \leq \beta $, where $h_T$ is the diameter of  $T$ and $\rho_T$ is the diameter of the largest ball contained in $T$. For every element $T$, we will denote by $\boldsymbol{n}_T$ the outward unit normal vector to $T$, writing $\boldsymbol{n}$ instead of $\boldsymbol{n}_T$ when there is no confusion. We will follow the standard convention and denote $\displaystyle h:=\max_{T\in \mc{T}_h} h_T$. For the sake of simplicity we assume $h<1$.

We will denote by $e$ any face of a simplex and will call it an \textit{interior face} if there are two elements $T^+$ and $T^-$ in $\mc{T}_h$ such that $e = \partial T^+ \cap \partial T^-$. The set of all such faces will be denoted by $\mc{E}_h^\circ$. Also, we say that  a face is a \textit{boundary face} if there is an element $T\in \mc{T}_h$ such that $e=\partial T \cap \Gamma_h$ and will denote the set of boundary faces by $ \mc{E}_h^{\partial}$. The entirety of the faces of the triangulation, $\mc{E}_h$, can then be decomposed as $\mc{E}_h= \mc{E}_h^\circ \cup \mc{E}_h^{\partial}$. The length of a face $e$ will be denoted by $h_e$.

The jump of a scalar-valued function across interior faces will be denoted by $\jump{w}:=
w^+-w^-$. At the boundary faces we set $\jump{w}:=w-\varphi_h$, where $\varphi_h$ is the approximation of the boundary data at $\Gamma_h$ that will be defined later in \eqref{def:varphi_h}. For a vector-valued function $\mbf{v}$, its jump across interior faces will be denoted by $\jump{\mbf{v}}:=\mbf{v}^+\cdot \mbf{n}^+ + \mbf{v}^-\cdot \mbf{n}^-$.

\paragraph{Spaces and norms.}\label{sec:SpacesNorms}
We utilize standard terminology
for Sobolev spaces and norms, where vector-valued functions and their corresponding spaces are denoted in bold face. In particular, if $\cO$ is a domain in 
$\mathbb{R}^{d}$, $\Sigma$ is an open or closed Lipschitz curve ($d=2$) or surface ($d=3$), and $s\in\mathbb{R}$, we define $\bH^{s}(\cO) := [\rmH^{s}(\cO)]^{d}$ and $\bH^{s}(\Sigma) := [\rmH^{s}(\Sigma)]^{d}$. However, when $s=0$ we write $\bL^{2}(\cO)$ and $\bL^{2}(\Sigma)$ instead of $\bH^{0}(\cO)$ and $\bH^{0}(\Sigma)$, respectively.
The associated norms are denoted by $\|\cdot\|_{s,\cO}$ and $\|\cdot\|_{s,\Sigma}$, writing simply $\|\cdot\|_{\cO}$ and $\|\cdot\|_{\Sigma}$ when $s=0$, and the corresponding $L^2$ inner products will be denoted by $(\cdot,\cdot)_{\mathcal O}$ and $\pdual{\cdot,\cdot}_{\partial \mathcal O}$. For $s\geq 0$, we write $|\cdot|_{s,\cO}$ for the $\bH^{s}$-semi norm and  $\rmH^{s}$-semi norm.

The mesh-dependent inner products for a triangulation $\mathcal T_h$ are given by
	\begin{equation*}
	(\cdot, \cdot)_{\mc{T}_h} := \sum_{T\in \mc{T}_h} (\cdot, \cdot)_T , \qquad \pdual{\cdot, \cdot}_{\partial \mc{T}_h} :=\sum_{T\in \mc{T}_h} \pdual{\cdot, \cdot}_{\partial T} \quad  \text{ and } \quad  \langle\cdot, \cdot\rangle_{\Gamma_h} := \sum_{e\in \mc{E}_h^{\partial}} \langle\cdot, \cdot\rangle_e ,
	\end{equation*}	
and their corresponding norms will be denoted, respectively,  by
    \begin{equation*}
    \| \cdot \|_{\compD} := \left( \sum_{T\in \mc{T}_h} \| \cdot \|_T^2 \right)^{1/2}, \qquad \| \cdot \|_{\partial \mc{T}_h} := \left( \sum_{T\in  \mc{T}_h} \| \cdot \|_{\partial T}^2 \right)^{1/2} \quad  \text{ and } \quad   \| \cdot \|_{\Gamma_h} := \left( \sum_{e\in \mc{E}_h^{\partial}} \| \cdot \|_e^2 \right)^{1/2} .
    \end{equation*}
To avoid proliferation of unimportant constants, we will use the terminology $a\lesssim b$ whenever $a\leq C b$ and $C$ is a positive constant independent of $h$. 
\paragraph{The extended domain.}\label{sec:ExtendedDomain}
Given a triangulation $\mc{T}_h$ and a boundary face $e \in \mc{E}_h^{\partial}$ we will denote by $T^e$ the unique element of $\mc{T}_h$ having $e$ as a face. To a point $\mbf{x} \in e$, we associate a point $\overline{\mbf{x}} \in \Gamma_h$ and set $l(\mbf{x})=|\mbf{x}-\overline{\mbf{x}}|$. If we let $\mbf{t}=\mbf{t}(\mbf{x})$ be a normalized vector in the direction connecting $\mbf{x}$ to $\overline{\mbf{x}}$ then we can parameterize the line segment between them by
    \begin{equation*}
	\sigma_{\mbf{t}}(\mbf{x}) := \{ \mbf{x} +s\mbf{t} , s\in [0, l(\mbf{x})]\},
	\end{equation*}
and define the \textit{extension patch} as
	\begin{equation*}
	T^e_{ext} := \{ \mbf{x} +s\mbf{t} : 0 \leq s \leq l(\mbf{x}) , \mbf{x} \in e \}.
	\end{equation*}
In principle, $\overline{\mbf{x}}$ can be specified in several ways. For instance, it can be a point that minimizes the distance between $\mbf{x}$ and $\Gamma_h$. However, in that case $\overline{\mbf{x}}$ might be not unique and also the union of all such extension patches $T^e_{ext}$ may not cover the set $\compD^c := \Omega \setminus \overline{\compD}$ entirely if $\Omega$ is not convex. A second possibility is to set $\overline{\mbf{x}}$ to be the closest intersection between $\Gamma$ and the ray starting at $\mbf{x}$ having tangent vector $\mbf{n}$, the normal to the face $e$ where $\mbf{x}$ belongs. In that case, $\mbf{t}=\mbf{n}$ and $T^e_{ext}$ may not cover $\compD^c$. Moreover, $l(\mbf{x})$ could be extremely large compared to the mesh size. To define the numerical method, we consider the algorithm proposed by \cite{CoSo2012} that constructs $\overline{\bf{x}}$ in such a way that three conditions are satisfied: $\overline{\bf{x}}$ is unique, two different line segments $\sigma_{\mbf{t}}$ in $T^e_{ext}$ do not intersect each other inside $T^e_{ext}$ and the line  $\sigma_{\mbf{t}}$ does not intersect the interior of $\Omega_h$. In this case, the union of $T^e_{ext}$ completely covers $\compD^c$.

\paragraph{The extension operator.}\label{sec:extension}
Now that an extension patch $T^e_{ext}$ has been defined so that for every $T^e_{ext}\in \Omega\setminus\overline{\Omega_h}$ there corresponds a single $T^e\in \mathcal T_h$, we can define a way to extend polynomial functions defined only in the computational domain. This will be needed when transferring the boundary condition to the computational domain $\Gamma_h$.

Let $p:T^e\to \mathbb R $ be a polynomial function. We will define its extension to $T^e_{ext}$ as
    \begin{equation}\label{eq:ExtensionOperator}
    \boldsymbol E_h(p)(\boldsymbol y) := p|_{T^e}(\boldsymbol y) \qquad  \forall \, \boldsymbol y \in T^e_{ext}.
    \end{equation}
We will keep notation simple and a polynomial function $p$ should be understood as its extrapolation $E_h(p)$  whenever an evaluation outside of $\Omega_h$ is required. This should be clear from the context. For vector-valued polynomial functions, the extension is defined similarly component by component.

\subsection{The HDG method}\label{sec:HDG}
We consider the finite dimensional spaces of piecewise polynomials
	\begin{align*}
	\mbf{V}_h &:= \{\mbf{v}\in \boldsymbol L^2(\mc{T}_h) : \mbf{v}|_T \in [\md{P}_k(T)]^2, \ \forall \ T \in \mc{T}_h \}, \\
	W_h &:= \{w\in L^2(\mc{T}_h) : w|_T \in \md{P}_k(T), \ \forall \ T \in \mc{T}_h \}, \\
	M_h &:= \{\mu\in L^2(\mc{E}_h) : \mu|_T \in \md{P}_k(e), \ \forall \ e \in \mc{E}_h \},
	\end{align*}
where, $\md{P}_k(T)$ denotes the space of polynomials of degree at most $k$ defined in $T\in \mc{T}_h$ and the space $\md{P}_k(e)$, for faces $e\in \mc{E}_h$, is similarly defined. 
The HDG scheme associated to \eqref{eq:mixed} reads: Find $(\mbf{q}_h, u_h, \what{u}_h) \in \mbf{V}_h\times W_h \times M_h$, such that
	\begin{subequations}\label{eq:HDG}
	\begin{align}
	(\kap^{-1} \mbf{q}_h, \mbf{v})_{\mc{T}_h} - (u_h, \nabla \cdot \mbf{v})_{\mc{T}_h} + \pdual{\what{u}_h, \mbf{v} \cdot \mbf{n}}_{\partial \mc{T}_h} &=  0, \label{eq:HDG_a} \\
	-(\mbf{q}_h, \nabla w)_{\mc{T}_h} + \pdual{ \what{\mbf{q}}_h \cdot \mbf{n},w}_{\partial \mc{T}_h} &=  (\mc{F}(u_h),w)_{\mc{T}_h}, \label{eq:HDG_b} \\
	\pdual{\what{u}_h,\mu}_{\Gamma_h} &=  \pdual{\varphi_h,\mu}_{\Gamma_h},  \label{eq:HDG_c} \\
	\pdual{\what{\mbf{q}}_h \cdot \mbf{n}, \mu}_{\partial \mc{T}_h \setminus \Gamma_h} &= 0 	\label{eq:HDG_d},
	\end{align}
for all $(\mbf{v}, w, \mu)\in \mbf{V}_h \times W_h \times M_h$. Here 
    \begin{eqnarray}\label{eq:HDG_e}
    \what{\mbf{q}}_h\cdot \mbf{n} := \mbf{q}_h \cdot \mbf{n} +  \tau(u_h - \what{u}_h) \quad \textrm{on}\quad \partial \mc{T}_h,
    \end{eqnarray}
with $\tau$ being a positive stabilization function, whose maximum will be denoted by $\overline{\tau}$, and the approximate boundary condition, motivated by \eqref{def: varphi}, is given by
	\begin{equation}\label{def:varphi_h}
	\varphi_h(\mbf{x}):=  g(\overline{\mbf{x}})+\int_0^{l(\mbf{x})} (\kap^{-1}\,\mbf{q}_h)(\mbf{x} + \boldsymbol{t}(\mbf{x})s) \cdot  \boldsymbol{t}(\mbf{x}) ds, \quad \textrm{for} \quad \mbf{x} \in \Gamma_h.
	\end{equation}
	\end{subequations}
Note that the function $\kappa$ is defined by \eqref{eq:mixed} in the full domain $\Omega$, while the flux $\boldsymbol q_h$ is extended from $\Omega_h$ to $\Omega$ as defined in \eqref{eq:ExtensionOperator}.
%
\subsection{Local post processing of the scalar solution}\label{sec:pp} 
Our a posteriori error estimator will be obtained in terms of a local post processing $u_h^*$, which approximates the scalar unknown $u$ with enhanced accuracy.  We seek for $u_h^*$ in the space 
    \begin{equation*}
    W_h^*:= \{w\in L^2(\mc{T}_h) : w|_T \in \md{P}_{k+1}(T), \ \forall \, T \in \mc{T}_h \},
    \end{equation*}
such that, in each element $T\in \mc{T}_h$, satisfies:
    \begin{subequations}\label{eq:post_processing}
    \begin{align}
    (\kap \nabla u_h^*, \nabla w)_T + (\mc{F}(u_h^*), w)_T &= -(\boldsymbol q_h,\nabla w)_T + (\mc{F}(u_h),w)_T  && \forall \, w \in \md{P}_{k+1}(T),\label{eq:post_processing1}\\
    (u_h^*,w)_T &= (u_h,w)_T  && \forall \, w\in \md{P}_0(T).\label{eq:post_processing2}
    \end{align}
    \end{subequations}
In the case where $\mc{F}$ is independent of $u$, it is well known (Section 5.2 in \cite{CoGoSa2010}) that $u_h^*$ is well defined and converges to $u$ with order $h^{k+2}$ when the solution has enough regularity. It is also known that there is a variety of choices to construct $u_h^*$. In fact, we could consider a simpler choice and use $(\kap \nabla u_h^*, \nabla w)_T = -(q_h,\nabla w)_T$ instead of \eqref{eq:post_processing1}. However, as we will see in Section \ref{sec:Aposteriori}, the term involving  $\mc{F}$ plays a key role in deriving the error estimator.

Consider real numbers $l_u,l_{\mbf{q}}\in[0,k]$ and assume that $u\in H^{l_u+2}(\mathcal{T}_h)$ and $\mbf{q}\in H^{l_{\mathbf{q}}+1}(\mathcal{T}_h)$. We can sate the following result on the well posedness and convergence rate of the post-processing. The proof of this statement makes use of some results derived in the forthcoming {\it a priori} error analysis and will be postponed to Appendix \ref{sec:AppendixB}. 
\begin{lem}\label{lem:estim post-proc}
The local post processing $u_h^*$ is well defined for $L$ small enough. Moreover, if $Lh^2<1$ and $k\geq 1$, then
    \begin{subequations}\label{eq:estim post-proc}
    \begin{alignat}{6}
      \|u-u_h^*\|_{\compD} &\lesssim (Rh)^{1/2} (h^{l_u+1} |u|_{l_u+2,\compD} + h^{l_u+1} |\mbf{q}|_{l_{\mbf{q}}+2,\compD} )+h^{l_u+2} |u|_{l_u+2,\compD}+L h^{l_u+1}
  |u|_{l_u+2,\compD}, \label{eq:estim post-proc_1}\\ 
    |u-u_h^*|_{1,T} &\lesssim\, h_T^{l_u+1} |u|_{l_u+1,T} + Lh_T\|\varepsilon^u \|_{0,T} +\|\boldsymbol q - \boldsymbol q_h\|_{0,T} + Lh_T\|u-u_h\|_{0,T} \label{eq:estim post-proc_2} \\
    \intertext{ and }
    \sum_{e\in \mathcal{E}_h^\partial} h_e^{1/2} \|\jump{u_h^*}\|_e &\lesssim \|u-u_h^*\|^{1/2}_{\compD}   \left( \|u-u_h^*\|_{\compD}^2 + h^2 |u-u_h^*|^2_{1,\compD} \right)^{1/4} \label{eq:estim post-proc_3}.
    \end{alignat}
    \end{subequations}
\end{lem}

Here, $R$---which will be defined properly in the following section---is proportional to the product $h^{-1} dist(\Gamma_h,\Gamma)$. When $dist(\Gamma_h,\Gamma)$ is of order $h$, then $R$ is of order one. This result guarantees a superconvergence of $h^{k+2}$ if $L<h$  and $R$ is of order $h$. If $R$ is of order one, it only ensures a convergence of order $h^{k+3/2}$. However, for the linear case, \cite{CoQiuSo2014,CoSo2012} reported numerical experiments suggesting that the order is indeed $h^{k+2}$ even when $R$ is of order one.

\section{Well-posedness}\label{sec:well-posedness}

In this section we employ a Banach fixed-point argument to ensure the well-posedness of \eqref{eq:HDG}. To that end, we define the operator $\mc{J} :  W_h  \to  W_h $ that maps $\zeta$ to the second component of the triplet $(\mbf{q}, u, \what{u})\in \mbf{V}_h\times W_h \times M_h$ satisfying the linearized HDG system \eqref{eq:HDG} where the source has been evaluated at $\zeta$, namely
	\begin{subequations}\label{eq:Fixed Point}
	\begin{align}
	(\kap^{-1}\ \mbf{q}, \mbf{v})_{\mc{T}_h} - (u, \nabla \cdot \mbf{v})_{\mc{T}_h} + \pdual{\what{u}, \mbf{v} \cdot \mbf{n}}_{\partial \mc{T}_h} &= 0, \label{eq:Fixed Point 1} \\
	-(\mbf{q}, \nabla w)_{\mc{T}_h} + \pdual{ \what{\mbf{q}} \cdot \mbf{n},w}_{\partial \mc{T}_h}  &= (\mc{F}(\zeta),w)_{\mc{T}_h}, \label{eq:Fixed Point 2} \\
	\pdual{\what{u},\mu}_{\Gamma_h} &=  \pdual{\varphiJ,\mu}_{\Gamma_h}, \label{eq:Fixed Point 3} \\
	\pdual{\what{\mbf{q}} \cdot \mbf{n}, \mu}_{\partial \mc{T}_h \setminus \Gamma_h} &= 0, \label{eq:Fixed Point 4}
	\end{align}

for all $(\mbf{v}, w, \mu)\in \mbf{V}_h \times W_h \times M_h$, where $\what{\mbf{q}}\cdot \mbf{n} := \mbf{q} \cdot \mbf{n} +  \tau(u - \what{u})$ and
    \begin{equation}\label{varphiq}
    \varphiJ(\mbf{x}):=  g(\overline{\mbf{x}})+\int_0^{l(\mbf{x})} (\kap^{-1}\,\mbf{q})(\mbf{x} + \boldsymbol{t}(\mbf{x})s) \cdot  \boldsymbol{t}(\mbf{x}) ds,
    \end{equation}
    	\end{subequations}
for $\mbf{x} \in \Gamma_h$. We stress the difference between the non-linear mapping $\mc J$, which maps \textit{arguments} of the source $\mc F(\cdot)$ to solutions of the corresponding HDG system, and the linear solution operator $\mc S: W_h\to W_h$ that maps the source term $\mc F$ itself to the solution of the corresponding HDG system.

As we can expect, assumptions on the line segments $\sigma_{\mbf{t}}$ and distance between $\Gamma_h$ and $\Gamma$ will be needed to ensure the well-posedness of \eqref{eq:Fixed Point} and the contraction property of the fixed-point operator. In order to specify these assumptions we need to define some additional quantities. Given a boundary face $e\in\mathcal E^\partial_h$, its corresponding element $T^e$, and the extension patch $T^e_{ext}$, we denote by $h_e^{\perp}$ (resp. $H_e^{\perp}$) the largest distance of a point inside $T_e$ (resp. $T_{ext}^e$) to the the plane determined by the face $e$, set $r_e := H_e^{\perp} / h_e^{\perp}$ and $\displaystyle R:= \max_{e\in \mc{E}_h^{\partial}} r_e$.  

For an extended patch $T^e_{ext}$ and element $T^e\in\mathcal T_h$ sharing a face $e$, we define 
\[
\mathcal V^k := \left\{ \boldsymbol p \in \mathbb [\mathbb{P}_k(T^e_{ext} \cup T^e)]^2\, : \,  \boldsymbol p\cdot \boldsymbol{n}_e\neq\boldsymbol 0 \right\}.
\]
and we denote by $\boldsymbol n_e$ the \textit{interior} normal vector to $T^e_{ext}$ along the face $e$---i.e. the exterior normal vector to $T^e$ pointing in the direction of $T^e_{ext}$. We can then introduce the constants
    \begin{equation*}
	C^e_{ext} := \dfrac{1}{\sqrt{r_e}} \sup_{\bsy{\chi}\in \mathcal V^k} \dfrac{\|\bsy{\chi}\cdot\boldsymbol n_e \|_{T_{ext}^e} }{\|\bsy{\chi}\cdot\boldsymbol n_e\|_{T^e}} \quad \text{ and } \quad
	C^e_{inv} := h_e^{\perp} \sup_{\bsy{\chi}\in \mathcal V^k} \dfrac{\|\nabla \bsy{\chi}\cdot \boldsymbol n_e  \|_{T^e} }{\|\bsy{\chi}\cdot\boldsymbol n_e \|_{T^e}},
	\end{equation*}
where, in abuse of notation, $\boldsymbol n_e$ is a constant vector field in $T^e_{ext}$  that coincides with the normal vector (pointing outwards $T^e$) associated to the face $e$. In \cite{CoQiuSo2014} it was shown that the constants $C^e_{ext}$ and $C^e_{inv}$ are independent of $h$, but depend on the polynomial degree; in particular, the supremum appearing in the definition of $C^e_{inv}$ is proportional to $(h_e^\perp)^{-1}$.

With these definitions in place, we can now state the following set of geometric assumptions on the boundary faces of the triangulation.

\textbf{Assumptions.} For each $e\in \mc{E}_h^{\partial}$ we will require the following to hold:
\begin{subequations}\label{eq:S}
\begin{multicols}{2}\noindent
	\begin{equation}
	\mbf{t}(\mbf{x}) = \mbf{n} \text{ for all }\mbf{x}\in e,\label{eq:S1} 
    \end{equation}
    \begin{equation}
    r_e\leq C,\label{eq:S2}
	\end{equation}
	\begin{equation}
	\displaystyle \overline{\tau}  \, H_e^{\perp}\, \underline{\kappa}^{-1} \leq \dfrac{1}{3},\label{eq:S3}
	\end{equation}
	\begin{equation}
	\overline{\kap}\, \underline{\kap}^{-1}\, r_e^3\, (C_{ext}^e\,  C_{inv}^e)^2  \leq 1 .\label{eq:S4}
	\end{equation}
\end{multicols}
\end{subequations}
Before proceeding, let us comment on this set of assumptions. As mentioned at the end of Section \ref{sec:Notation}, the vector $\mbf{t}(\mbf{x})$ does not necessarily have to be normal to the face $e$. Therefore, the results presented in what follows still hold if \eqref{eq:S1} is not satisfied as long as the difference $1 -\mbf{t}(\mbf{x}) \cdot \mbf{n}$ is positive and small enough.  However, this assumption helps us to facilitate the presentation of the ideas behind the proofs. On the other hand, \eqref{eq:S2} imposes the geometric constraint that the family of triangulations $\{\mathcal T_h\}$ should be such that the distance between the computational boundary $\Gamma_h$ and the true boundary $\Gamma$ remains locally of the same order of magnitude as the face mesh parameter $h_e$. Moreover, it guarantees that as long as $H_e^\perp >0$, then $H_e^\perp\sim h_e^\perp$; if $H_e^\perp = 0$ for some $e$, then no transfer of boundary conditions is needed on that particular face---as this would only happen if $e\cap\Omega=e$. Given that the stabilization parameter $\tau$ is of order one, \eqref{eq:S3} states that the minimum value of the diffusion coefficient $\kappa$ imposes a restriction on how far apart $\Gamma_h$ and $\Gamma$ are allowed to be.  Due to the proportionality guaranteed by \eqref{eq:S2}, then \eqref{eq:S3} will hold whenever the mesh size---and therefore the distance between the boundaries---is small enough. Assumption \eqref{eq:S4} is the most demanding of all. By requiring $r_e$ to be small enough compared to the product $\overline{\kappa}\,\underline{\kappa}^{-1}$, the condition effectively limits the range of values of $\kappa$ that the method is able to resolve for a given, fixed, distance between the computational and physical boundaries, as measured by $H_e^\perp$. Not surprisingly, the closer to zero the diffusion coefficient gets, the smaller $H_e^\perp$ must be with respect to the mesh size near the computational boundary.
	
The main result of this section, Theorem \ref{thm:Fixed Point}, is that under suitable assumptions $\mc{J}$ is a contraction and therefore the solution of \eqref{eq:HDG} can be obtained by applying it iteratively. The proof of this fact is almost a straightforward consequence of the linearity of the solution map and of a key stability bound established in Lemma \ref{lem:StabilityS} that estimates the norm of $u$ in terms of those of the sources and the boundary conditions. However, the proof of the latter follows from a lengthy series of estimates. In order to prioritize clarity of exposition, we first present this estimate without proving it and show how the main result follows from it. The technical details of the proof of Lemma \ref{lem:StabilityS} are then presented afterwards. 
\begin{lem}\label{lem:StabilityS} Suppose that Assumptions \eqref{eq:S1} throughout \eqref{eq:S4} and the elliptic regularity of the auxiliary problem \eqref{eq: dual problem} are satisfied. Then, there exists $\widetilde{c}>0$, independent of $h$ such that 
    \begin{equation}\label{eq:stability}
	\|u\|_{\compD} \leq 4\, \max\{\widetilde{c}^2 \, h, 1\} \|\mc{F}(\zeta)\|_{\compD} + 2\, \widetilde{c} (\sqrt{3}+1) \, h^{1/2} \|\kap^{1/2}\, l^{-1/2}\,  \g\|_{\Gamma_h},
	\end{equation}
where $\g(\mbf{x})=\mathrm{g}(\overline{\mbf{x}}(\mbf{x}))  \ \forall\,  \mbf{x}\in \Gamma_h$.	
\end{lem}
Thanks to this estimate the main result, from which well-posedness of the problem follows, can be proved in a very compact way, as we now demonstrate.
\begin{thm}\label{thm:Fixed Point}
If Assumptions \eqref{eq:S1} throughout \eqref{eq:S4} and the elliptic regularity of the auxiliary problem \eqref{eq: dual problem} hold, then $\mc{J}$ is well-defined. Furthermore, if we assume $4\, L\, \max\{\widetilde{c}^2\, h, 1\} < 1 $, then $\mc{J}$ is a contraction operator.
\end{thm}
\begin{proof}
The system in \eqref{eq:Fixed Point} is linear and has a unique solution under the set of assumptions \eqref{eq:S} (see \cite{CoQiuSo2014}), therefore the operator $\mc{J}$ is well-defined.

Let $\zeta_1, \zeta_2 \in W_h$ and consider $u_1 = \mc{J}(\zeta_1)$ and $u_2 = \mc{J}(\zeta_2)$. Then, $u_1$ and $u_2$ are the second component of the solution of \eqref{eq:Fixed Point} with right hand sides $\mc{F}(\zeta_1)$ and $\mc{F}(\zeta_2)$, respectively. Hence, the difference $u_1-u_2$ satisfies equations \eqref{eq:Fixed Point}, with source term $\mc{F}(\zeta_1)-\mc{F}(\zeta_2)$ and homogeneous boundary conditions on $\Gamma$. By the stability estimate in Lemma \ref{lem:StabilityS} and Lipschitz continuity assumption, we have
    \begin{equation*}
	\|\mc{J}(\zeta_1)-\mc{J}(\zeta_2)\|_{\compD} = \| u_1-u_2 \|_{\compD} \leq 4\max\{\widetilde{c}^2\, h, 1\} \|\mc{F}(\zeta_1)-\mc{F}(\zeta_2)\|_{\compD}
	\leq 4\,L\, \max\{\widetilde{c}^2 \,h, 1\} \|\zeta_1-\zeta_2\|_{\compD}. 	
    \end{equation*}
The result follows due to  $4\, L\, \max\{\widetilde{c}^2\, h, 1\} < 1 $. 
\end{proof}
As a consequence of the above result, system \eqref{eq:HDG} subject to the hypotheses of the theorem has a unique solution that depends continuously on the problem data. 

We now present the arguments that lead to the proof of Lemma \ref{lem:StabilityS}. We start by establishing a connection between the norm of the transferred boundary conditions $\varphiJ$, the magnitude of the flux $\boldsymbol q$ and the length of the transfer path taken. In order to do so, we will make use of a tool introduced in \cite{CoQiuSo2014}. For any smooth enough function $\mbf{v}$ defined in $T^e\cup T_{ext}^e$ and $\mbf{x}\in\Gamma_h$ we set
	\begin{equation}\label{def:delta}
	\delta_{\mbf{v}} (\mbf{x}) := \dfrac{1}{l(\mbf{x})} \int_0^{l(\mbf{x})} [\mbf{v}(\mbf{x} + \mbf{n}s) - \mbf{v}(\mbf{x}) ] \cdot \mbf{n} ds.
	\end{equation}
The significance of this function is that it will allow us to separate the contributions to the boundary conditions coming from the flux, from the diffusivity,  and from the length of the transfer path. Observe that due to Assumption \eqref{eq:S1} and \eqref{varphiq}, we have
	\begin{align*}
	\varphiJ(\mbf{x}) - \mathrm{g}(\overline{\mbf{x}}(\mbf{x})) &=  \int_0^{l(\mbf{x})} \kap^{-1}(\mbf{x})\, \mbf{q}(\mbf{x} + \mbf{n}s) \cdot \mbf{n} ds \\
	&= \kap^{-1}(\mbf{x})\int_0^{l(\mbf{x})} [ \mbf{q}(\mbf{x} + \mbf{n}s) - \mbf{q}(\mbf{x})] \cdot \mbf{n} ds + l(\mbf{x}) \kap^{-1}(\mbf{x}) \mbf{q}(\mbf{x}) \cdot \mbf{n} ,
	\end{align*}
with $\mbf{q} \in \mbf{V}_h$, and using the definition of $\delta_{\mbf{q}}$, given in \eqref{def:delta}, we can rewrite the above as
	\begin{equation}\label{eq:varphiq}
	\varphiJ(\mbf{x}) - \mathrm{g}(\overline{\mbf{x}}(\mbf{x})) =  \kap^{-1}(\mbf{x}) \,l(\mbf{x}) \,\delta_{\mbf{q}} (\mbf{x}) +  \kap^{-1}(\mbf{x})\, l(\mbf{x}) \, \mbf{q}(\mbf{x}) \cdot \mbf{n}.
	\end{equation}
This expression, combined with the bounds that we will derive in Lemma \ref{lem:estim-varphiq} below, will enable us to estimate the approximate solution in terms of the sources, as will become evident in Lemma \ref{lem:EstimateNormH}. 

In proving the next result, we will have to make use of the following properties of $\delta_{\boldsymbol v}$, which hold for each $e\in \mc{E}_h^{\partial}$ \cite[Lemma 5.2]{CoQiuSo2014}:
	\begin{subequations}\label{ineq:deltav}
	\begin{align}
	\| l^{1/2} \delta_{\mbf{v}} \|_e &\leq \dfrac{1}{\sqrt{3}} \, r_e^{3/2}\, C_{ext}^e\, C_{inv}^e\, \|\mbf{v}\|_{T^e}  & &\forall \, \mbf{v} \in [\md{P}_k(T)]^d, \label{ineq:delta Pk} \\
	\| l^{1/2}\delta_{\mbf{v}} \|_e &\leq \dfrac{1}{\sqrt{3}}\, r_e \,\| h_e^{\perp} \partial_n \mbf{v} \cdot \mbf{n}\|_{T_{ext}^e} & &\forall \, \mbf{v} \in [H^1(T)]^d. 	\label{ineq:delta H1}
	\end{align}
	\end{subequations}
The following three inequalities follow readily from estimate \eqref{ineq:delta Pk}, assumptions \eqref{eq:S4} and \eqref{eq:S3}, and Young's inequalities.

\begin{lem}\label{lem:estim-varphiq}
Let $\varphiJ$ be the transferred boundary condition defined in \eqref{varphiq} and suppose that Assumptions S are satisfied. It holds
	\begin{align*}
	|\pdual{ \varphiJ , \delta_{\mbf{q}} }_{\Gamma_h}| &\leq \dfrac{1}{6} \|\kap^{1/2}l^{-1/2} \varphiJ\|^2_{\Gamma_h} + \dfrac{1}{2} \|\kap^{-1/2}\mbf{q}\|_{\compD}^2 \\
	|\pdual{\varphiJ, \tau(u - \what{u})}_{\Gamma_h}| &\leq \dfrac{1}{6} \|\kap^{1/2}l^{-1/2}\varphiJ\|^2_{\Gamma_h} + \dfrac{1}{2} \|\tau^{1/2} (u - \what{u})\|^2_{\partial \mc{T}_h} \\
   |\langle \varphiJ, \kap \, l^{-1}\, \g\rangle_{\Gamma_h}| &\leq \dfrac{1}{6} \|\kap^{1/2}l^{-1/2}\,  \varphiJ\|^2_{\Gamma_h} + \dfrac{3}{2} \|\kap^{1/2}l^{-1/2}\,  \g\|^2_{\Gamma_h}
	\end{align*}
	\hfill$\square$
\end{lem}

The expression for $\varphiJ$ in \eqref{eq:varphiq} implies that $\mbf{q}(\mbf{x}) \cdot \mbf{n} =  \kap(\mbf{x}) l^{-1}(\mbf{x}) (\varphiJ(\mbf{x}) - \g)  - \delta_{\mbf{q}} (\mbf{x})$. Thus, thanks to the definition of $\what{\mbf{q}} \cdot \mbf{n}$, it follows that
	\begin{equation}\label{eq:q hat}
	\what{\mbf{q}} \cdot \mbf{n} = \kap\, l^{-1} (\varphiJ -\g )  - \delta_{\mbf{q}} + \tau (u - \what{u}) \qquad \text{ on } \Gamma_h.
	\end{equation}
The above expression can now be combined with the estimates from Lemma \ref{lem:estim-varphiq} to produce a bound for the norm of $(\mbf{q}, u-\what{u}, \varphiJ)$ in terms of the source $\mc{F}(\zeta)$ and the boundary data $\g$ as we will show next. 
\begin{lem}\label{lem:EstimateNormH} 
If Assumptions S hold, then
    \begin{equation*}
\triple{(\mbf{q}, u-\what{u}, \varphiJ)}^2 \leq 2\|\mc{F}(\zeta)\|_{\compD} \|u\|_{\compD}  +3 \|\kap^{1/2} l^{-1/2} \g \|_{\Gamma_h}^2,        
    \end{equation*}
where 
    \begin{align}\label{def:normH}
     \triple{(\mbf{v},w,\mu)} &:= \left( \| \kap^{-1/2} \mbf{v} \|_{\compD}^2 + \|\tau^{1/2} w\|_{\partial \mc{T}_h}^2 + \|\kap^{1/2} l^{-1/2} \mu \|_{\Gamma_h}^2 \right)^{1/2}.
    \end{align}	
\end{lem}
\begin{proof}
Take $\zeta \in W_h$ and let $u=\mc{J}(\zeta)\in W_h$ be the corresponding solution satisfying \eqref{eq:Fixed Point}. Integrating by parts the left hand side in \eqref{eq:Fixed Point 2},  testing \eqref{eq:Fixed Point} with $\mbf{v}=\mbf{q}$, $w=u$, and
    \begin{equation*}
    \mu:= \left\{ \begin{array}{ccl}
	 - \what{\mbf{q}} \cdot \mbf{n} &  &\text{on } \Gamma_h, \\
	 - \what{u}&  &\text{on } \partial \mc{T}_h \setminus \Gamma_h  ,
	 \end{array} 
	 \right.
    \end{equation*}
and adding the resulting equalities, we get
    \begin{equation*}
    \|\kap^{-1/2}\, \mbf{q} \|^2_{\compD} + \|\tau^{1/2}\, (u-\what{u})\|^2_{\partial \mc{T}_h} = (\mc{F}(\zeta), u)_{\mc{T}_h} - \langle \varphiJ, \what{\mbf{q}}\cdot \mbf{n} \rangle_{\Gamma_h}. 
    \end{equation*}
Then, using the fact that $\what{\mbf{q}} \cdot \mbf{n} = \mbf{q} \cdot \mbf{n} + \tau (u - \what{u})$ in combination with identity \eqref{eq:q hat}, we obtain 
\[
    	\triple{(\mbf{q}, u-\what{u}, \varphiJ)}^2 \leq \|\mc{F}(\zeta)\|_{\compD}  \|u\|_{\compD} + |\pdual{\varphiJ, \delta_{\mbf{q}}}_{\Gamma_h} | + |\pdual{\varphiJ,\tau (u - \what{u})}_{\Gamma_h}| + | \langle \varphiJ, \kap\, l^{-1}\, \g) \rangle_{\Gamma_h} |.
\]	
By the Cauchy-Schwarz inequality and Lemma \ref{lem:estim-varphiq}, we arrives at
	\begin{align*}
    	\triple{(\mbf{q}, u-\what{u}, \varphiJ)}^2 \leq& \|\mc{F}(\zeta)\|_{\compD}  \|u\|_{\compD} +  \dfrac{1}{2} \|\kap^{1/2}l^{-1/2} \varphiJ\|^2_{\Gamma_h} 
    	+ \dfrac{1}{2} \|\kap^{-1/2}\mbf{q}\|_{\compD}^2 \\
    	&+ \dfrac{1}{2} \|\tau^{1/2} (u - \what{u})\|^2_{\partial \mc{T}_h} 
    	+\dfrac{3}{2} \|\kap^{1/2}l^{-1/2}\,  \g\|^2_{\Gamma_h}.
	\end{align*}
and the result follows.
\end{proof}
%
\paragraph{Proof of Lemma \ref{lem:StabilityS}.} 
%
With all the previous technical results in place we are now in a position to prove the crucial result. In the arguments below, $\boldsymbol \Pi_{\boldsymbol V}$ and $\Pi_W$ are, respectively, the $\boldsymbol V $ and $W$ components of the HDG projector introduced in the Appendix \ref{sec:HDGprojection}.	

\begin{proof}
We will proceed using an auxiliary problem  that generalizes the result of Lemma 3.3 in \cite{CoQiuSo2014} to our semi-linear case. We will consider that, given $\Theta \in L^2(\Omega)$, the solution to the auxiliary problem
	\begin{subequations}\label{eq: dual problem}
	\begin{align}
	&& \kap^{-1} \bsy{\phi}+ \nabla \psi &= 0 & &\text{in } \Omega, \label{eq: dual problem 1} \\
	&& \nabla \cdot \bsy{\phi} &= \Theta & &\text{in } \Omega, \label{eq: dual problem 2} \\
	&& \psi &= 0 & &\text{on } \partial \Omega ,\label{eq: dual problem 3}
	\end{align}
	\end{subequations}
satisfies the regularity estimate
	\begin{equation}\label{regularity dual problem}
	\| \bsy{\phi} \|_{H^1(\Omega)} + \| \psi \|_{H^2(\Omega)} \leq C_{reg} \|\Theta\|_{\Omega}.
	\end{equation}	
Then, following the argument of \cite[Lemma 4.1]{CoGoSa2010} it is possible to show that if $u$ satisfies \eqref{eq:Fixed Point} then
 \[
	(u, \Theta)_{\mc{T}_h} =  ( \kap^{-1}\ \mbf{q}, \bsy{\Pi}_{\mbf{V}} \bsy{\phi} - \bsy{\phi})_{\mc{T}_h} + \pdual{\what{u}, \bsy{\phi} \cdot \mbf{n}}_{\Gamma_h} - \pdual{\what{\mbf{q}} \cdot \mbf{n} , \psi}_{\Gamma_h} + (\mc{F}(\zeta), \Pi_W \psi)_{\mc{T}_h}.
\]
We will now use this expression to bound the norm of $u$. In order to simplify the exposition, we will group some terms on the right hand side of this expression and treat them separately.  Hence, we decompose the above expression as
    \begin{equation}\label{eq:uDotTheta}
    (u, \Theta)_{\mc{T}_h} = \md{T}_{\boldsymbol q} + \md{T}_{\mc{F}} + \md{T}_{u},
    \end{equation}
by defining
    \begin{equation*}
    \md{T}_{\boldsymbol q} := ( \kap^{-1}\ \mbf{q}, \bsy{\Pi}_{\mbf{V}} \bsy{\phi} - \bsy{\phi})_{\mc{T}_h}, \quad  \md{T}_{u} := \pdual{\what{u}, \bsy{\phi} \cdot \mbf{n}}_{\Gamma_h} - \pdual{\what{\mbf{q}} \cdot \mbf{n} , \psi}_{\Gamma_h} , \quad \text{ and } \quad \md{T}_{\mc{F}} := (\mc{F}(\zeta), \Pi_W \psi)_{\mc{T}_h}.
    \end{equation*}
The terms $\md{T}_{\mbf{q}}$ and $\md{T}_{\mathcal F}$ can be bounded by an application of the estimates \eqref{error_projector1} in combination with the elliptic regularity \eqref{regularity dual problem}, yielding
    \begin{align}\label{eq:estim_Tq_priori}
    |\md{T}_{\mbf{q}}| &\leq \underline{\kap}^{-1/2} \, \|\kap^{-1/2}\, \mbf{q}\|_{\compD} \|\bsy{\Pi}_{\mbf{V}} \bsy{\phi} - \bsy{\phi}\|_{\compD} \lesssim  \,h\, \|\kap^{-1/2}\, \mbf{q}\|_{\compD}  \|\Theta\|_{\Omega}, 
    \end{align}
and
    \begin{align}\label{eq:estim_Tf_priori}
    |\md{T}_{\mc{F}}|  &\lesssim   \|\mc{F}(\zeta)\|_{\compD}\, \|\Theta\|_{\Omega}. 
    \end{align}
The treatment of the term  $\md{T}_{u}$ requires more work. Denoting by $Id_M$ the identity operator in  $M$, considering  \eqref{eq:q hat} and equation \eqref{eq: dual problem 1}, $\md{T}_{u}$ can be written as $\md{T}_{u} = \sum_{i=1}^5 \md{T}_{u}^i$, where 
    \begin{equation*}
    \begin{array}{lll}
    \md{T}_{u}^1 := - \langle \kap l^{-1} \varphiJ ,\psi + l \partial_n \psi \rangle_{\Gamma_h}, & \quad &  \md{T}_{u}^2 := \langle\kap \varphiJ, (P_M - Id_M)\partial_n \psi \rangle_{\Gamma_h}, \\
    \md{T}_{u}^3 := \langle \delta_{\mbf{q}}, \psi \rangle_{\Gamma_h}, & \quad & \md{T}_{u}^4 := -\langle \tau(u - \what{u}) ,P_M \psi \rangle_{\Gamma_h}, \\
    \md{T}_{u}^5 := \langle\kap\,  l^{-1} \, \g, \psi \rangle_{\Gamma_h}. & \quad &
\end{array}
    \end{equation*}
We will now determine bounds for all the terms in the decomposition.

By Young's inequality and combining the fact that $l(\boldsymbol{x})\lesssim R h \,\,, \forall \, \mbf{x} \in \Gamma_h$  with estimate \eqref{est 3 para T_zh }, we have
    \begin{equation*}
	|\md{T}_{u}^1| \leq \left| \langle \kap^{1/2}\, l\, \kap^{1/2} \,l^{-1/2} \,\varphiJ , l^{-3/2}\, (\psi + l \partial_n \psi) \rangle_{\Gamma_h} \right| 
	\lesssim 
	\,  R \, h \,  \|\kap^{1/2} \,l^{-1/2}\, \varphiJ \|_{\Gamma_h} \|\Theta \|_{\Omega}.
    \end{equation*}
Analogously, we get
    \begin{equation*}
	|\md{T}_{u}^2| \lesssim 	
	\, R^{1/2}\, h\, \|\kap^{1/2} l^{-1/2} \varphiJ \|_{\Gamma_h} \|\Theta \|_{\Omega}.        
    \end{equation*}
To bound $\md{T}_{u}^3$, we employ \eqref{ineq:delta Pk}, \eqref{est 4 para T_zh }, and \eqref{eq:S4} yielding
	\begin{align*}
	|\md{T}_{u}^3|
	&\lesssim (R\, h)^{1/2} \| l^{1/2}\, \delta_{\mbf{q}} \|_{\Gamma_h} \|l^{-1} \,\psi\|_{\Gamma_h} \\ 
	&\lesssim (R\, h)^{1/2} \left( \sum_{e\in \mc{E}_h^{\partial}}  \dfrac{1}{3} r_e^{3}\,  (C_{ext}^e \,C_{inv}^e )^2 \|\mbf{q}\|^2_{T^e} \right)^{1/2}  \|\Theta\|_{\Omega}\\
	&\lesssim
    \,  R^2\, h^{1/2}\, \|\kap^{-1/2}\mbf{q}\|_{\compD} \|\Theta\|_{\Omega}.
	\end{align*}
Similarly, using \eqref{est 4 para T_zh } we can bound
\[
|\md{T}_{u}^4| \lesssim  \overline{\tau}^{1/2\,} R\, h\, \|\tau^{1/2}\, (u -\what{u})\|_{\partial \mc{T}_h} \|\Theta\|_{\Omega} \quad \text{ and } \quad |\md{T}_{u}^5| \lesssim
(Rh)^{1/2} \|\kap^{1/2}\, l^{-1/2}\, \g\|_{\Gamma_h} \|\Theta\|_{\Omega}.
\]
Taking $\Theta = u$ in $\compD$ and $\Theta = 0$ in $\compD^c$ in \eqref{eq: dual problem} and considering the bounds for the terms $\md{T}_{u}^i$, we can combine the decomposition \eqref{eq:uDotTheta} with the estimates \eqref{eq:estim_Tq_priori} and  \eqref{eq:estim_Tf_priori}, to obtain
\begin{align*}
	\|u\|_{\compD} 
	&\lesssim  h\,  \|\kap^{-1/2} \mbf{q}\|_{\compD} +   h\,  (R+R^{1/2} )\,  \| \kap^{1/2}\, l^{-1/2} \varphiJ \|_{\Gamma_h} + R^2\,   h^{1/2} \| \kap^{-1/2}\mbf{q}\|_{\compD} \\
	&\qquad + \overline{\tau}^{1/2}\,  R\, h\,  \|\tau^{1/2}(u - \what{u})\|_{\partial \mc{T}_h}  + \|\mc{F}(\zeta)\|_{\compD} + (Rh)^{1/2} \|\kap^{1/2}\, l^{-1/2}\, \g\|_{\Gamma_h} . 
	\end{align*}
Now, let     
    \begin{equation}\label{eq:alpha}
    \widetilde{c} := C\, \left\{  1+ R^2 +R+R^{1/2}+ \overline{\tau}^{1/2} \, R \right\},
    \end{equation}    
where $C>0$ is the constant hidden in the symbol $\lesssim$.
Then, since $h<1$ by Lemma \ref{lem:EstimateNormH} and Young's inequality, we infer
	\begin{align*}
	\|u\|_{\compD} &\leq  \widetilde{c} \, h^{1/2} \left( \sqrt{2}\,  \|\mc{F}(\zeta)\|_{\compD}^{1/2}  \|u \|_{\compD}^{1/2} + \sqrt{3}\,   \|\kap^{1/2}\, l^{-1/2}\,  \g\|_{\Gamma_h} \right)  + \|\mc{F}(\zeta)\|_{\compD} + \widetilde{c}\,  h^{1/2}\,   \|\kap^{1/2}\, l^{-1/2} \, \g\|_{\Gamma_h} \\
	&\leq \widetilde{c}^2 \,h \,\|\mc{F}(\zeta)\|_{\compD}  + \dfrac{1}{2} \|u\|_{\compD}  + \|\mc{F}(\zeta)\|_{\compD} +  \widetilde{c}\, (\sqrt{3}+ 1) \,h^{1/2} \,\|\kap^{1/2}\, l^{-1/2}\,  \g\|_{\Gamma_h}
	\end{align*}
and thus
    \begin{equation*}
	\|u\|_{\compD} \leq 4\, \max\{\widetilde{c}^2 \, h, 1\} \|\mc{F}(\zeta)\|_{\compD} + 2\, \widetilde{c} \, (\sqrt{3}\, +1) h^{1/2} \|\kap^{1/2}\, l^{-1/2}\,  \g\|_{\Gamma_h}. 
    \end{equation*}
\end{proof}
%
\section{\textit{A priori} error analysis}
%
We now provide the \textit{a priori} error bounds for the method. As we will see, some of the results presented in this section can be proven by using similar arguments to those of Section \ref{sec:well-posedness} and many details will be omitted. Given that the set of  assumptions \eqref{eq:S} is required to hold in order to ensure the well-posedness of the problem, in the present section they will be assumed as true and used for the error analysis without explicitly stating them in the results. Similarly, the regularity assumption \eqref{regularity dual problem} will be assumed to hold.

The total approximation error has a component due to the accuracy of the discretization, and a component due entirely to the approximation properties of the discrete subspace. This is made apparent using the HDG projection defined in \eqref{eq:HDGprojector} and defining  \textit{the projections of the errors}
\[
\bsy{\varepsilon}^{\mbf{q}}: = \bsy{\Pi}_{\mbf{V}}\mbf{q} - \mbf{q}_h \quad \text{ and } \quad \varepsilon^u:= \Pi_W u - u_h,
\]
and the \textit{error of the projections}
\[
\mbf{I}^{\mbf{q}}:= \mbf{q}-\bsy{\Pi}_{\mbf{V}}\mbf{q}  \quad \text{ and } \quad I^u:=u - \Pi_W u.
\]
Using these quantities we can decompose the error as follows
\[
\mbf{q}-\mbf{q}_h = \bsy{\varepsilon}^{\mbf{q}}+\mbf{I}^{\mbf{q}} \quad \text{ and } \quad u-u_h =  \varepsilon^u + I^u.
\]
In addition, we define $\varepsilon^{\what{u}}:= P_M u - \what{u}_h$, where we recall that $P_M$ is the $L^2$ projection into $M_h$. 

It is not difficult to show that $( \errorq ,\erroru ,\errortu )$ belongs to $\mbf{V}_h\times W_h \times M_h$ and satisfies 
	\begin{subequations}\label{eq:error projection}
	\begin{align}
	(\kap^{-1} \errorq, \mbf{v})_{\mc{T}_h} - (\erroru, \nabla \cdot \mbf{v})_{\mc{T}_h} + \langle \errortu, \mbf{v} \cdot \mbf{n} \rangle _{\partial \mc{T}_h} &=  -(\kap^{-1}  \mbf{I}^{\mbf{q}},\mbf{v})_{\mc{T}_h}, \label{eq:error projection_a}  \\
	-(\errorq, \nabla w)_{\mc{T}_h} + \langle \errortq \cdot \mbf{n},w \rangle_{\partial \mc{T}_h} &=  (\mc{F}(u)-\mc{F}(u_h),w)_{\mc{T}_h}, \label{eq:error projection_b}\\
	\langle \errortu,\mu \rangle_{\Gamma_h} &=  \langle \varphi - \varphi_h ,\mu \rangle_{\Gamma_h},  \label{eq:error projection_c}\\
	\langle\errortq \cdot \mbf{n}, \mu \rangle_{\partial \mc{T}_h \setminus \Gamma_h} &= 0, \label{eq:error projection_d}
	\end{align}
	\end{subequations}
for all $(\mbf{v}, w, \mu)\in \mbf{V}_h \times W_h \times M_h$, where $\errortq \cdot \mbf{n} := \errorq \cdot \mbf{n} +  \tau(\erroru - \errortu ) = P_M(\mbf{q} \cdot \mbf{n}) - \what{\mbf{q}}_h\cdot \mbf{n}$. This \textit{error equations} will help us establishing two results that will eventually lead to the proof of the convergence of the method.

To abbreviate the notation in the following arguments it will be useful to define
    \begin{subequations}\label{def:lambda}
    \begin{eqnarray}
    \projerrorq &: =&  \left( \| \mbf{I}^{\mbf{q}}\|^2_{\compD} +  \| h^{\perp}\, \partial_n (\mbf{I}^{\mbf{q}} \cdot \mbf{n})\|^2_{\Omega_h^c}  + \| (h^{\perp})^{1/2}\, \mbf{I}^{\mbf{q}} \cdot \mbf{n})\|^2_{\Gamma_h} \right)^{1/2}, \label{def:lambda_q}\\
    \projerroru &: =&  \left( \|(h^{\perp})^{1/2} \,I^u \|_{\Gamma_h}  + \|I^u\|_{\Omega_h} \right)^{1/2}. \label{def:lambda_u}
    \end{eqnarray}
    \end{subequations}
With respect to these quantities we point out that, if $\mbf{q}\in \mbf{H}^{k+1}(\Omega)$, $u\in H^{k+1}(\Omega)$ and $\tau = \mathcal O(1)$ then, by scaling arguments and the properties \eqref{eq:projection_error}, both $\projerrorq$ and $\projerroru$ are of order $h^{k+1}$.

The first of these auxiliary lemmas establishes the convergence of the discrete flux $\boldsymbol q_h$, the restriction to the mesh skeleton $\widehat u_h$, and the transferred boundary data $\varphi_h$ as a consequence of the convergence of the primary scalar variable $u_h$ and the errors of the projections $I^u$ and $\boldsymbol I^{\boldsymbol q}$.

\begin{lem}\label{lem:estim-priori error}
Let $\triple{\cdot}$ be the norm defined in \eqref{def:normH}. There exists a positive constant C>0, independent of $h$, such that
	\begin{equation}\label{ineq:estim-priori error}
	\triple{ (\errorq ,\erroru - \errortu, \varphi - \varphi_h)}^2 \leq 4\,  L (\| \erroru\|_{\Omega_h} + \|I^u\|_{\Omega_h})\, \|\erroru\|_{\compD} + C\,  \projerrorq^2.
	\end{equation}	
\end{lem}
\begin{proof}
Testing \eqref{eq:error projection} with
$	\mbf{v}: = \errorq $,  $w:= \erroru$ and $\mu:= \left\{ \begin{array}{rcl}
	 - \errortq \cdot \mbf{n} &  &\text{on } \Gamma_h \\
	 -\erroru&  &\text{on } \partial \mc{T}_h \setminus \Gamma_h  
	 \end{array} 
	 \right.
	$,
results in
\[
    \|\kap^{-1/2}\, \errorq \|^2_{\compD} + \| \tau^{1/2}\, (\erroru- \errortu) \|_{\partial \mc{T}_h} = - (\kap^{-1}\, \mbf{I}^{\mbf{q}}, \errorq)_{\mc{T}_h} + (\mc{F}(u) - \mc{F}(u_h), \erroru)_{\mc{T}_h} - \langle \varphi - \varphi_h, \errortq \cdot \mbf{n} \rangle_{\Gamma_h},
\]
then, owing to \eqref{eq:q hat}, we readily obtain $\errortq \cdot \mbf{n} = \kap \, l^{-1}\, (\varphi-\varphi_h) - \delta_{\errorq} - \delta_{\mbf{I}^{\mbf{q}}} - \mbf{I}^{\mbf{q}} \cdot \mbf{n} + \tau(\erroru - \errortu)$ on $\Gamma_h$. Substituting this above, we get 
    \begin{equation}\label{ineq:estim-priori error 1}
        \begin{array}{c}
        \triple{(\errorq, \erroru - \errortu, \varphi-\varphi_h)}^2  \leq |(\kap^{-1}\, \mbf{I}^{\mbf{q}},  \errorq)_{\mc{T}_h} | + L (\|\erroru\|_{\compD} + \|I^u\|_{\compD} ) \, \|\erroru\|_{\compD}  \\
        + |\langle \varphi- \varphi_h,  \delta_{\errorq} + \delta_{\mbf{I}^{\mbf{q}}} + \mbf{I}^{\mbf{q}} \cdot \mbf{n} - \tau(\erroru - \errortu) \rangle_{\Gamma_h}| .
    \end{array}
    \end{equation}
The estimates in  Lemma \ref{lem:estim-varphiq}, can be applied to the last term of \eqref{ineq:estim-priori error 1} to arrive at
	\begin{equation}\label{ineq:estim-priori error aux}
	\begin{array}{rcl}
	\langle \varphi - \varphi_h, \delta_{\mbf{I}^{\mbf{q}}} \rangle_{\Gamma_h}  &\leq& \displaystyle \tfrac{1}{6}\,  \| \kap^{1/2}\, l^{-1/2} \, (\varphi - \varphi_h)\|^2_{\Gamma_h} + \tfrac{1}{2}\,  \underline{\kappa}^{-1}  \, \max_{e\in \mathcal{E}_h^\partial} \{ r_e^2\}  \| h^{\perp} \partial_n (\mbf{I}^{\mbf{q}} \cdot \mbf{n})\|^2_{\Omega_h^c} , \\
     \langle \varphi - \varphi_h, \delta_{\errorq} \rangle _{\Gamma_h} &\leq& \tfrac{1}{4} \| \kap^{1/2}\, l^{-1/2}\, ( \varphi - \varphi_h)\|^2_{\Gamma_h} + \tfrac{1}{3} \,  \| \kap^{-1/2}\, \errorq \|^2_{\Omega_h}, \\
    \langle \varphi - \varphi_h, \mbf{I}^{\mbf{q}} \cdot \mbf{n}  \rangle_{\Gamma_h} &\leq& \displaystyle \tfrac{1}{6}\,  \| \kap^{1/2}\, l^{-1/2}\, ( \varphi - \varphi_h)\|^2_{\Gamma_h} + \tfrac{3}{2}\,  \underline{\kappa}^{-1}\, \max_{e\in \mathcal{E}_h^\partial} \{r_e\} \, \| (h^{\perp})^{1/2} \, \mbf{I}^{\mbf{q}} \cdot \mbf{n})\|^2_{\Gamma_h} , \\
   |\langle \varphi - \varphi_h, \tau(\erroru - \errortu)\rangle _{\Gamma_h}|  &\leq&  \tfrac{1}{6} \, \| \kap^{1/2\, }l^{-1/2} \, ( \varphi - \varphi_h)\|^2_{\Gamma_h} + \tfrac{1}{2} \, \| \tau^{1/2}\, (\erroru - \errortu)\|^2_{\partial \mc{T}_h}. 
   \end{array}
	\end{equation}
The estimate  \eqref{ineq:estim-priori error} is obtained with $C:= 4\, \kap^{-1}\, \max \{1, \frac{1}{2}\, R^2, \frac{3}{2} \, R\}$,  applying Young's inequality to term $| (\kap^{-1}\, \mbf{I}^{\mbf{q}}, \errorq)_{\mc{T}_h}|$ and the estimates given in \eqref{ineq:estim-priori error aux} .
\end{proof}

Due to the previous result, it is enough to show the convergence of $\erroru$ to guarantee the convergence of the method. The next step then is to estimate $\|\varepsilon^u\|_{\Omega_h}$, which we will do through a duality argument very much in the spirit of the proof of Lemma \ref{lem:StabilityS}. Indeed, given $\Theta \in L^2(\Omega)$, and considering the linear auxiliary problem \eqref{eq: dual problem}, but now using equations \eqref{eq:error projection} instead of \eqref{eq:Fixed Point}, we can decompose 
	\begin{equation}\label{eq:dual_identity}
	(\erroru, \Theta)_{\mc{T}_h} = \md{T}^{\mc{F}} + \md{T}^{\mbf{q}} + \md{T}^{u}
	\end{equation}
where
\begin{align*}
\md{T}^{\mc{F}}:=\,& (\mc{F}(u) - \mc{F}(u_h), \Pi_W \psi)_{\mc{T}_h}, \\
\md{T}^{\boldsymbol q} :=\,&(\kap^{-1} (\mbf{q}-\mbf{q}_h), \bsy{\Pi}_{\mbf{V}} \bsy{\phi})_{\mc{T}_h} + (\errorq,\nabla \psi )_{\mc{T}_h}, \\
\md{T}^{u} :=\,& \langle \errortu, \bsy{\phi}\cdot \mbf{n}\rangle_{\Gamma_h} -\langle \errortq \cdot \mbf{n}, \psi\rangle_{\Gamma_h}.
\end{align*}
In order to estimate the size of $\erroru$ we will now treat each of these terms separately. The term $\md{T}^{\mc{F}}$ is easy to bound, since
    \begin{equation}\label{eq:BoundErrorComponentA}
    |\md{T}^{\mc{F}}| \leq L \|u - u_h\|_{\Omega_h} \|\Pi_W \psi\|_{\Omega_h} \leq L\left(\|\varepsilon^u \|_{\Omega_h} + \|I^u\|_{\Omega_h}\right)\|\Pi_W \psi\|_{\Omega_h} \lesssim L\left(\|\varepsilon^u \|_{\Omega_h} + \|I^u\|_{\Omega_h}\right)\|\Theta\|_{\Omega}.
    \end{equation}
Now, by adding and subtracting  $ (\kap^{-1}(\mbf{q}-\mbf{q}_h), \bsy{\phi} )_{\mc{T}_h}$ in the definition of the term $\md{T}^{\boldsymbol q}$, we obtain
    \begin{equation*}
    \md{T}^{\boldsymbol q} = (\kap^{-1}(\mbf{q}-\mbf{q}_h), \bsy{\Pi}_{\mbf{V}} \bsy{\phi} -\bsy{\phi})_{\mc{T}_h} + (\kap^{-1}(\mbf{q}-\mbf{q}_h), \bsy{\phi} )_{\mc{T}_h}+ (\bsy{\varepsilon}^{\mbf{q}},\nabla \psi )_{\mc{T}_h}.
    \end{equation*}
However, due to \eqref{eq: dual problem 1}, it holds that $(\kap^{-1}(\mbf{q}-\mbf{q}_h), \bsy{\phi})_{\mc{T}_h}+ (\bsy{\varepsilon}^{\mbf{q}},\nabla \psi )_{\mc{T}_h}=-(\mbf{I}^{\mbf{q}},\nabla \psi)_{\mc{T}_h}$. Let $\psi_h\in W_h$ be arbitrary. Then, by \eqref{properties projector Pi_w}, we have $(\mbf{I}^{\mbf{q}},\nabla \psi_h ) =0$. Combining these last two facts we obtain
	\begin{equation*}
		\md{T}^{\boldsymbol q} = (\kap^{-1}(\mbf{q}-\mbf{q}_h), \bsy{\Pi}_{\mbf{V}} \bsy{\phi} -\bsy{\phi})_{\mc{T}_h} + (\mbf{I}^{\mbf{q}},\nabla ( \psi -\psi_h))_{\mc{T}_h}.
    \end{equation*}
Therefore, by choosing $\psi_h=\Pi_W \psi$, it follows that
    \begin{align}
    \nonumber
		|\md{T}^{\boldsymbol q}|\leq\,& \|\kappa ^{-1/2}(\errorq + \boldsymbol I^{\boldsymbol q})\|_{\Omega_h}\|\bsy{\Pi}_{\mbf{V}} \bsy{\phi} -\bsy{\phi}\|_{\Omega_h} + \|\boldsymbol I^{\boldsymbol q}\|_{\Omega_h}\|\nabla ( \psi -\psi_h)\|_{\Omega_h} \\
		\label{eq:tq}
		\lesssim\,& h^{\min\{1,k\}} \|\kappa ^{-1/2}\errorq\|_{\compD} \|\Theta\|_{\Omega}+h^{\min\{1,k\}} \|\mbf{I}^{\mbf{q}}\|_{\compD} \|\Theta\|_{\Omega}
	\end{align}
where we have used the elliptic regularity for the auxiliary problem, the approximation properties \eqref{eq:HDGprojector} and \eqref{eq:projection_error} of the HDG projector. 
Finally, we can further decompose $\md{T}^{u} := \sum_{i=1}^7 \md{T}^{u}_i$, where:
    \begin{equation*}
    \begin{array}{lll}
	\md{T}^{u}_1 := -\langle \kap l^{-1} (\varphi - \varphi_h), \psi + l \partial_n \psi \rangle_{\Gamma_h}, & \quad &
	\md{T}^{u}_2 := -\pdual{\kap(\varphi - \varphi_h), (Id_M - P_M)\partial_n \psi }_{\Gamma_h}, \\
	\md{T}^{u}_3 := \pdual{\delta_{\mbf{I}^q}, \psi}_{\Gamma_h}, & \quad &
	\md{T}^{u}_4 := \pdual{\mbf{I}^{\mbf{q}} \cdot \mbf{n}, (Id_M-P_M) \psi }_{\Gamma_h}, \\
	\md{T}^{u}_5 := -\pdual{\tau P_M I^u, \psi}_{\Gamma_h}, & \quad &
	\md{T}^{u}_6 := \pdual{\delta_{\bsy{\varepsilon}^{\mbf{q}} }, \psi}_{\Gamma_h}, \\
	\md{T}^{u}_7 := -\langle\tau(\erroru - \errortu), P_M\psi \rangle _{\Gamma_h}. & \quad &
	\end{array}
    \end{equation*}
Bounding separately each of the terms above it is possible to estimate $\|\varepsilon^u\|_{\mathcal T_h}$, as we show below.

\begin{lem}\label{lem:estim Eu}
Assume that the Lipschitz constant is such that $L$ is small enough, and consider the discrete spaces to be of polynomial degree $k\geq 1$. Then, 
\begin{align}\label{eq:estim Eu}
	\|\erroru\|_{\compD} \lesssim  ((Rh)^{1/2}( 1 + \overline{\tau}^{1/2})+h) \triple{(\bsy{\varepsilon}^{\mbf{q}}, \varepsilon^u - \varepsilon^{\what{u}} ,\varphi - \varphi_h)} 
	+ (Rh^{1/2}+L) (\projerrorq +  \projerroru).
	\end{align}
\end{lem}
\begin{proof}
By applying Young's inequality to each term in the decomposition of $\md{T}^{u}$, considering the estimates in Lemma \ref{lem:aux Tu},  using the fact $l(\boldsymbol{x})\lesssim R h \,\,, \forall \boldsymbol{x} \in \Gamma_h$ and having in mind the estimates in \eqref{ineq:deltav}, it is possible to deduce:
    \begin{equation*}
    \begin{array}{lll}
	|\md{T}^{u}_1| \lesssim \overline{\kap}^{1/2}\,  R\,  h\, \| \kap^{1/2}\, l^{-1/2} \, ( \varphi - \varphi_h) \|_{\Gamma_h} \|\Theta\|_{\Omega}, & \quad &
	|\md{T}^{u}_5| \lesssim \overline{\tau} \,R \, h^{1/2} \|(h^{\perp})^{1/2} I^u \|_{\Gamma_h} \|\Theta\|_{\Omega}, \\
	|\md{T}^{u}_2| \lesssim \overline{\kap}^{1/2} \, R^{1/2} \, h \, \| \kap^{1/2}\, l^{-1/2} \, ( \varphi - \varphi_h) \|_{\Gamma_h} \|\Theta\|_{\Omega}, & \quad &
	|\md{T}^{u}_6| \lesssim \overline{\kappa}^{1/2} \, R^2\, h^{1/2} \, \| \kap^{-1/2}\,  \errorq  \|_{\Omega_h} \|\Theta\|_{\Omega},\\
	|\md{T}^{u}_3| \lesssim  R^{3/2}\, h^{1/2} \|h^{\perp} \partial_n \mbf{I}^{\mbf{q}} \cdot \mbf{n} \|_{\Gamma_h} \|\Theta\|_{\Omega}, & \quad &
	|\md{T}_{u}^7| \lesssim  \overline{\tau}^{1/2}  R h \| \tau^{1/2}(\erroru - \errortu ) \|_{\Gamma_h} \|\Theta\|_{\Omega}.\\
	|\md{T}^{u}_4| \lesssim h \|(h^{\perp})^{1/2} \mbf{I}^{\mbf{q}} \cdot \mbf{n} \|_{\Gamma_h} \|\Theta\|_{\Omega}, &&
\end{array}
\end{equation*}
Then, recalling the definition of the norm $\triple{\cdot}$ in \eqref{def:normH}, and of the terms $\projerrorq$ and $\projerroru$  in \eqref{def:lambda}, we get
    \begin{equation}\label{eq:estim Eu 1}
    \begin{array}{rl}
	|\md{T}^{u}| &\lesssim  \bigg(\overline{\kappa}^{1/2} \, R h + \overline{\kappa}^{1/2}\, R^{1/2}\, h + \overline{\kappa}^{1/2} R^2\, h^{1/2} + \overline{\tau}^{1/2}  R h \bigg) \triple{(\bsy{\varepsilon}^{\mbf{q}}, \varepsilon^u - \varepsilon^{\what{u}} ,\varphi - \varphi_h)} \,\,  \|\Theta\|_{\Omega}\\
	&\quad + \max\{R^{1/2}, \overline{\tau}\, \}\, Rh^{1/2}\, (\projerroru + \projerrorq) \, \|\Theta\|_{\Omega}
 + h\projerrorq \|\Theta\|_{\Omega}.
	\end{array} 
	\end{equation}	
Finally, taking $\Theta = \erroru$ in $\compD$ and $\Theta = 0$ in $\compD^c$ in \eqref{eq:dual_identity} and using the estimates  \eqref{eq:BoundErrorComponentA}, \eqref{eq:tq} and  \eqref{eq:estim Eu 1}, and considering assumption \eqref{eq:S3}, the  estimate \eqref{eq:estim Eu} is obtained.
\end{proof}
	
Combining Lemmas \ref{lem:estim-priori error} and \ref{lem:estim Eu}, we can bound the error in terms of the error of the projection $I^u$ and $\boldsymbol I^{\boldsymbol q}$ as we do below.	
	
\begin{thm}\label{thm:estim normH eroor}
Assume that $6\, L\, \Big( (R\, h)^{1/2} \, (1+\overline{\tau}^{1/2}) + h\Big)  < 1$, $\tau$ is of order one, and the discrete spaces are of polynomial degree $k\geq 1$, then 
    \begin{subequations}
	\begin{eqnarray}
    \triple{(\bsy{\varepsilon}^{\mbf{q}}, \varepsilon^u - \varepsilon^{\what{u}} ,\varphi - \varphi_h)} \lesssim \projerrorq + \projerroru \label{ineq:estim normH eroor} 	
	\end{eqnarray}
and 
	\begin{eqnarray}\label{ineq:Eu}
	\|\erroru\|_{\compD} \lesssim \left((Rh)^{1/2}+L+h\right) (\projerroru + \projerrorq).
	\end{eqnarray}
 \end{subequations}	
\end{thm}

\begin{proof}
It follows from Lemma \ref{lem:estim-priori error} and   the estimate in \eqref{eq:estim Eu}, that
    \begin{align*}
	&\triple{(\bsy{\varepsilon}^{\mbf{q}}, \varepsilon^u - \varepsilon^{\what{u}} ,\varphi - \varphi_h)}^2 \leq 6\, L\, \|\erroru\|_{\compD}^2 + 2\, L\,  \projerroru^2 + C \, \projerrorq^2 \\
    &\quad \leq 6\, L\, \Big( (R\, h)^{1/2} \, (1+\overline{\tau}^{1/2}) + h \Big) \, \triple{(\bsy{\varepsilon}^{\mbf{q}}, \varepsilon^u - \varepsilon^{\what{u}} ,\varphi - \varphi_h)}^2 + \max\{6\, L(R\, h^{1/2} + L, C, 2\, L) \} \, (\projerroru^2 + \projerrorq^2),
    \end{align*}
where $C$ is the constant defined in Lemma \ref{lem:estim-priori error}. Then, due to $ 6\, L\, \Big( (R\, h)^{1/2} \, (1+\overline{\tau}^{1/2}) + h \Big) < 1$, the estimate \eqref{ineq:estim normH eroor} is fulfilled. Finally, \eqref{eq:estim Eu} and  \eqref{ineq:estim normH eroor} imply \eqref{ineq:Eu}.
\end{proof}   

As a byproduct of the previous result, we are now in the position to establish the asymptotic convergence rate of the discretization. 

\begin{crl}\label{cor:estim error converg}
Suppose that assumptions of Theorem \ref{thm:estim normH eroor} hold. If $u\in H^{k+1}(\Omega)$ and $\mbf{q} \in  \mbf{H}^{k+1}(\Omega)$, then
    \begin{eqnarray}\label{ineq:estim error converg}
    \|\mbf{q}-\mbf{q}_h\|_{\Omega} + \|u-u_h\|_{\Omega} \lesssim h^{k+1} \left( |u|_{k+1,\Omega} + |\mbf{q}|_{k+1,\Omega}  \right).
    \end{eqnarray}
\end{crl}
\begin{proof}
It follows from Theorem \ref{thm:estim normH eroor}, Lemma \ref{eq:estim Eu}, and the approximation properties \eqref{eq:projection_error},  combined with Lemmas 3.7 and 3.8 in \cite{CoQiuSo2014}.
\end{proof}
%
\section{A posteriori error analysis}\label{sec:Aposteriori}
%
\paragraph{A residual-based error estimator.}\label{sec:ErrorEstimator.}
%

In order to prevent the proliferation of high order (with respect to $h$) oscillatory terms that would only make the analysis more cumbersome, we will suppose for the remainder of this sections that $\varphi$ is the trace of a function in $W_h^* \cap H^1(\Omega_h)$. Let $(\mc{T}_h, \mc{E}_h^{\circ}, \mc{E}_h^{\partial})$ refer to the elements, interior faces and boundary faces of the computational mesh respectively. In each element $T\in \mc{T}_h$, we define the following residual-type local error estimator:
    \begin{equation}\label{def:local estimator}
    \begin{array}{rl}
    \eta_T(\mbf{q}_h , u_h^*, \varphi_h) : =& \bigg(  h_T^2 \|P_W \mc{F}(u_h^*) - \nabla \cdot \mbf{q}_h\|_T^2
    +  \|\kap^{1/2} \nabla u_h^* + \kap^{-1/2}\mbf{q}_h \|_{T}^2   \\
    &+ \displaystyle \sum_{e \in \mathcal{E}^\circ \cap T} \bigg(  h_e \|\jump{\mbf{q}_h}\|_e^2 + 
    h_e^{-1} \| \jump{u_h^*} \|_e^2 \bigg) + \sum_{e \in \mathcal{E}^\partial\cap T} h_e^{-1} \|  (\varphi_h- u_h^*) \|_e^2 \bigg)^{1/2},
    \end{array}
    \end{equation}
and introduce the data oscillation term
    \begin{align}\label{def:oscilator term}
    \text{osc}^2(\mc{F},\mc{T}_h) &:= \sum_{T\in \mc{T}_h} h_T^2 \| \mc{F}(u^*_h) - P_W \mc{F}(u^*_h) \|^2_T.
    \end{align}
We will show that the global error estimator, given by
    \begin{equation}\label{def:global estimator}
    \eta := \left( \sum_{T \in \mc{T}_h} \eta_T^2(\mbf{q}_h, u_h^*, \varphi_h)   \right)^{1/2},
    \end{equation}
constitutes a reliable and efficient local \textit{a posteriori} estimator for the error
    \begin{equation}\label{def:error posteriori}
    \aposteerror^2 :=\|\kap^{-1/2}(\mbf{q}-\mbf{q}_h)\|_{\compD}^2 + \|u - u_h^*\|_{\Omega_h}^2 + \|h_e^{-1/2}\, (\varphi - \varphi_h)\|_{\Gamma_h}^2.
    \end{equation}
The remaining part of this section is devoted to proving one of the main contributions of this work, which is the efficiency and reliability of the local error estimator \eqref{def:local estimator}. We state the result here, and will proceed to develop the tools required for its proof. This will follow readily from Theorem \ref{thm:reliable} and  Theorem \ref{thm:efficiency}, the proof of which is lengthy and requires a few technical lemmas.

\textbf{Reliability and local efficiency.}\textit{
If the Lipschitz constant $L$ associated to the source term $\mathcal F$ and the distance between $\Gamma_h$ and $\Gamma$ is small enough (in a sense that will be made clear in the hypothesis of Theorem \ref{thm:reliable}), then the error estimator $\eta$ is reliable, i.e.,
\[
\aposteerror^2 \lesssim \eta^2 + \text{osc}^2(\mc{F},\mc{T}_h).     
\]
Moreover, $\eta$ is locally efficient, meaning that
\[
\eta^2_T(\mbf{q}_h , u_h^*, \varphi_h) \lesssim  \sum_{T\in \mc{U}_h(e)} \left( \|\kap^{-1/2} \, (\mbf{q} - \mbf{q}_h) \|^2_T + \|u-u_h^*\|_T^2\right) + h_e^{-1} \|\varphi-\varphi_h\|^2_e +  \text{osc}^2(\mc{F},\mc{U}_h(e))
\]
where $\mc{U}_h(e)$ is the set of elements that have $e$ as an face. Namely, $\mc{U}_h(e) := \{ T \in \mathcal T_h : T \cap \mathcal E_h = e \}$.
}

Before setting out to show the validity of this result, we would like to make a few remarks regarding the steps required for the proof. The efficiency of the estimator can be established by adapting some of the arguments in \cite{aposteriori2} to account for the semi-linearity of the problem and for the approximation of the boundary data due to the curved boundaries. This will be addressed at the end in Theorem \ref{thm:efficiency}. Reliability, however, requires a much lengthier argument and the proof will be divided in several steps. Lemma \ref{lem:prop in W_h} establishes the connection between the residual of the HDG equation \eqref{eq:HDG_a} and the post-processed solution $u^*_h$ that appears in the local error estimator. To show that each of the terms in the estimator are indeed upper bounds for the error in the flux, the term $\|\kap^{-1/2}(\mbf{q}- \mbf{q}_h)\|^2_{\compD}$ will be decomposed in four components which will be treated separately in Lemma \ref{lem:q-qh}. This shows that the error on the flux can be successfully bounded by the estimator, plus some additional terms involving the error the scalar variable $u$ and the data in the boundary $\varphi$. Next, Lemma  \ref{lem:uh-uh*} shows an estimation for the error in the variable $u$ in terms of that of the transferred boundary condition and the flux. All these results are then consolidated in Theorem \ref{thm:reliable}, which establishes that the error can be controlled by a combination of the estimator $\eta$ and the data oscillation, that is, the reliability is finally proven.

As mentioned in Section \ref{sec:pp}, the inclusion of the nonlinear source term  $\mathcal{F}$ in the definition of $u_h^*$ helps obtaining the following result, which is important for the estimate in Lemma \ref{lem:q-qh}, that will link the post processed solution with Equation \eqref{eq:HDG_b}.
\begin{lem}\label{lem:prop in W_h}
Let $(u_h,\boldsymbol q_h)$ be the solutions to \eqref{eq:HDG} and $u_h^*$ be the post-processing of $u_h$ given by \eqref{eq:post_processing}. It holds
	\begin{equation*}
	(P_W \mc{F}(u^*_h)- \nabla \cdot \mbf{q}_h, w )_{\mc{T}_h} = - \pdual{\mbf{q}_h \cdot \mbf{n}, w}_{\partial \mc{T}_h} - (\kap \nabla u_h^*+\mbf{q}_h , \nabla w)_{\mc{T}_h} , \quad \forall \ w \in W_{1,h}^c, 
	\end{equation*}
where $W_{1,h}^c := \{w\in H_0^1(\Omega) : w|_T\in \mathbb P_1(T), \ \forall \ T \in \mc{T}_h \}$, and $\mathbb P_1(T)$ is the space of piecewise linear polynomials on $T$.
\end{lem}

\begin{proof}
Considering $w\in W_{1,h}^c$ and integrating by parts in the equation \eqref{eq:HDG_b} we obtain, for all $T\in \mc{T}_h$ 
\[
    (\nabla \cdot \mbf{q}_h, w)_{T}  - \pdual{\mbf{q}_h \cdot \mbf{n}, w}_{\partial T} + \pdual{ \what{\mbf{q}}_h \cdot \mbf{n},w}_{\partial T} =  (\mc{F}(u_h),w)_{T}.
\]
Then, due to $\pdual{ \what{\mbf{q}}_h \cdot \mbf{n},w}_{\partial \mc{T}_h} = 0$ and using \eqref{eq:post_processing}, we have
\[
    (\nabla \cdot \mbf{q}_h, w)_{\mc{T}_h}  - \pdual{\mbf{q}_h \cdot \mbf{n}, w}_{\partial \mc{T}_h} = (\kap \nabla u_h^*, \nabla w)_{\mc{T}_h} + (\mc{F}(u_h^*), w)_{\mc{T}_h} + (\mbf{q}_h,\nabla w)_{\mc{T}_h} ,
\]
which concludes the proof.
\end{proof}

In what follows, $ \wtil{u}^*_h\in W_h^*\cap H^1(\compD)$ such that $\wtil{u}^*_h=\varphi$ on $\Gamma_h$, will be used to denote the so-called Oswald interpolation of $u^*_h$ defined in \ref{lem: Oswald Interpolation}.  Now, we  apply the Lemma \ref{lem: Oswald Interpolation}, with $|\alpha|=1$, to get
    \begin{align*}
   \| \nabla(\wtil{u}^*_h - u_h^* ) \|^2_{\compD}
    &\leq C_O \, \left(\sum_{e\in \mc{E}_h^\circ} h_e^{-1} \| \jump{u_h^*} \|^2_e + \sum_{e\in \mc{E}_h^{\partial}} h_e^{-1} \| \varphi - u_h^*\|^2_e \right) \\
    &\leq  C_O  \left(\sum_{e\in \mc{E}_h^\circ} h_e^{-1} \| \jump{u_h^*} \|^2_e 
     +  2 \sum_{e\in \mc{E}_h^{\partial}} h_e^{-1} \| \varphi - \varphi_h\|^2_e + 2 \sum_{e\in \mc{E}_h^{\partial}} h_e^{-1} \| \varphi_h - u_h^*\|^2_e \right).
    \end{align*} 
where $C_O>0$ is a constant independent of $h$ arising from the approximation properties of the Oswald's interpolant. Similarly, for $|\alpha|=0$, we have
\[
\| \wtil{u}^*_h - u_h^*  \|^2_{\compD}
 \leq C_O   \left(\sum_{e\in \mc{E}_h^\circ} h_e \| \jump{u_h^*} \|^2_e 
     +  2\sum_{e\in \mc{E}_h^{\partial}} h_e \| \varphi - \varphi_h\|^2_e + 2\sum_{e\in \mc{E}_h^{\partial}} h_e \| \varphi_h - u_h^*\|^2_e\right) .
\]
Since for a fine enough mesh $h_e\leq h_e^{-1}$, the two inequalities above can be combined into
    \begin{equation}\label{eq:uh-uhtilde}
        \| \wtil{u}^*_h - u_h^*  \|^2_{\compD}\! +  \| \nabla(\wtil{u}^*_h - u_h^*)  \|^2_{\compD} \leq 2 C_O\! \left(   \sum_{e\in \mc{E}_h^\circ} h_e^{-1} \| \jump{u_h^*} \|^2_e   +  \sum_{e\in \mc{E}_h^{\partial}} h_e^{-1} \| \varphi_h - u_h^*\|^2_e +  \| h_e^{-1/2}(\varphi - \varphi_h)\|^2_{\Gamma_h} \right).
    \end{equation}
The following three results allow us to find a preliminary estimate for each term of our error defined in \eqref{def:error posteriori}. We begin rewrite  $\|\kap^{-1/2}(\mbf{q}- \mbf{q}_h)\|^2_{\mc{T}_h}$ in a suitable manner. Note first that using \eqref{eq:mixed_a}, and adding and subtracting $ \wtil{u}^*_h$ , it follows
	\begin{align*}
	\|\kap^{-1/2}(\mbf{q}- \mbf{q}_h)\|_T^2 &= (\mbf{q}-\mbf{q}_h, \kap^{-1}(\mbf{q} - \mbf{q}_h))_T  \\
	&= -(\mbf{q} - \mbf{q}_h, \nabla(u- \wtil{u}^*_h))_T - (\mbf{q}-\mbf{q}_h, \nabla \wtil{u}^*_h - \kap^{-1} \mbf{q}_h)_T \\
	&= (\nabla \cdot (\mbf{q} - \mbf{q}_h), u- \wtil{u}^*_h)_T - \pdual{(\mbf{q} - \mbf{q}_h)\cdot \mbf{n},u- \wtil{u}^*_h}_{\partial T} - (\mbf{q}-\mbf{q}_h, \nabla  \wtil{u}^*_h - \kap^{-1} \mbf{q}_h)_T.
	\end{align*}
Adding and subtracting $\mathcal F(u^*_h)$ and $P_W\mathcal F(u_h^*)$ in the first term above, and using \eqref{eq:mixed_b} to replace $\nabla\cdot\boldsymbol q$ by $\mathcal F(u)$ yields
	\begin{align*}
	\|\kap^{-1/2}(\mbf{q}- \mbf{q}_h)\|_T^2 &= ( \mc{F}(u^*_h) - P_W \mc{F}(u^*_h) + P_W \mc{F}(u^*_h) - \nabla \cdot \mbf{q}_h  + \mc{F}(u) - \mc{F}(u^*_h) , u-\wtil{u}^*_h)_T \\
	&\quad - \pdual{(\mbf{q} - \mbf{q}_h)\cdot \mbf{n},u-\wtil{u}^*_h}_{\partial T} - (\mbf{q}-\mbf{q}_h, \nabla \wtil{u}^*_h - \kap^{-1} \mbf{q}_h)_T.
	\end{align*}
Thus, since  $\mbf{q}\in H(\mathrm{div};\Omega_h)$ and $u-\wtil{u}^*_h\in H_0^1(\Omega_h)$, we can add over the entire grid to obtain  
\begin{align}\label{q_qh}
	\|\kap^{-1/2}(\mbf{q}- \mbf{q}_h)\|^2_{\compD} :=  \sum_{i=1}^4 \md{T}_{i},
 \end{align}
where 
\begin{equation*}
\begin{array}{lcl}
	\md{T}_{1} := (\mc{F}(u^*_h) - P_W \mc{F}(u^*_h), u - \wtil{u}^*_h)_{\mc{T}_h}&,& 
	\md{T}_{3} := -(\mbf{q} - \mbf{q}_h, \nabla \wtil{u}^*_h - \kap^{-1}\mbf{q}_h)_{\mc{T}_h}\\
	\md{T}_{2} := (P_W \mc{F}(u^*_h) - \nabla \cdot \mbf{q}_h, u - \wtil{u}^*_h)_{\mc{T}_h} + \langle \mbf{q}_h \cdot \mbf{n},u-\wtil{u}^*_h \rangle_{\partial \mc{T}_h}  &,&
	\md{T}_{4} := (\mc{F}(u) - \mc{F}(u^*_h), u-\wtil{u}^*_h)_{\mc{T}_h}.
\end{array}
\end{equation*} 
In the following estimates, for a given function $v$, let $Q_k(v)$ be the averaged Taylor polynomial of degree $k$ associated to $v$. For smooth functions this polynomial coincides with the ``usual'' Taylor polynomial, whereas for functions with Sobolev regularity it is defined by mollification of the weakly defined Taylor polynomial \cite[Section 4.1]{BrSc2008}.
\begin{lem}\label{lem:q-qh}
There exists $ \overline{C}_1>0$, independent of $h$ such that
    \begin{equation*}
    \begin{array}{rl}
   	&\|\kap^{-1/2}(\mbf{q}-\mbf{q}_h)\|^2_{\compD}
   	\leq \displaystyle  \overline{C}_1\, \left( \text{osc}^2(\mc{F},\mc{T}_h) + 	 \sum_{T\in \mc{T}_h} h_T^2 \|P_W \mc{F}(u^*_h) - \nabla \cdot \mbf{q}_h\|^2_{T} \right. \\
	&\qquad \left. + \displaystyle 
     \sum_{T\in \mc{T}_h} \|\kap^{1/2} \nabla u_h^* + \kap^{-1/2}\mbf{q}_h\|^2_{T} +   \sum_{e\in \mc{E}_h^\circ} \left( h_e\,  \|\jump{\mbf{q}_h}\|^2_{e} +  h_e^{-1}\, \| \jump{u_h^*} \|^2_e\right) + \sum_{e\in \mc{E}_h^{\partial}} h_e^{-1} \| \varphi_h - u_h^*\|^2_e  \right)  \\[2ex]
    &\qquad  \displaystyle +  \overline{C}_1\,   \| h_e^{-1/2}\, (\varphi - \varphi_h)\|^2_{\Gamma_h} +  \overline{C}_1\, L\,  \|u - u^*_h\|^2_{\compD}.
	\end{array}
	\end{equation*}
\end{lem}

\begin{proof}
To prove the result, we will bound each of the terms $ \md{T}_{i}$ in the decomposition separately. The final result will come as a consequence of the individual estimates. In some cases we will make use of a free parameter $\epsilon_j>0$.
%
\paragraph{Bound for $\md{T}_{1}$.}
%
Consider $Q_0(u-\wtil{u}^*_h)$, the zeroth order averaged Taylor polynomial associated to $u-\wtil{u}^*_h$. Since $(\mc{F}(u^*_h) - P_W \mc{F}(u^*_h), Q_0(u-\wtil{u}^*_h) )_T=0$, then by Young's inequality and the Bramble-Hilbert lemma with constant $c>0$, independent of $h$, we obtain
\[
	(h_T (\mc{F}(u^*_h) - P_W \mc{F}(u^*_h)), h_T^{-1}(u - \wtil{u}^*_h - Q_0(u-\wtil{u}^*_h)) )_T
	 \leq\dfrac{h_T^2}{4 \epsilon_1} \| \mc{F}(u^*_h) - P_W \mc{F}(u^*_h) \|^2_{T} +c\,  \epsilon_1\,  \|\nabla(u - \wtil{u}^*_h)\|^2_{T}.
\]
Using \eqref{eq: GS - comput domain_a} in the last term of the above expression to replace $\nabla u$ by $\kappa^{-1}\boldsymbol q$ along with adding and subtracting $ \nabla u^*_h$ and  $\kap^{-1} \mbf{q}_h$, we obtain
	\begin{align}\label{eq:est_aux_T1}
	 \|\nabla(u -\wtil{u}^*_h)\|^2_{T} 
\leq \frac{3}{\underline{\kap} } \left( \|\kap^{-1/2}(\mbf{q} - \mbf{q}_h)\|^2_{T} + \|\kap^{1/2} \nabla u^*_h + \kap^{-1/2} \mbf{q}_h \|_{T}^2 + \|\kap^{1/2} \nabla u^*_h -\kap^{1/2} \nabla \wtil{u}^*_h\|_{T}^2 \right).
	\end{align}
Thus
	\begin{align}
    \nonumber
    | \md{T}_{1}| \leq\,& \quad \displaystyle \what{C}_1 \epsilon_1 \sum_{T\in \mc{T}_h}  \left( \|\kap^{-1/2}(\mbf{q} - \mbf{q}_h)\|^2_{T}   +  \|\kap^{1/2} \nabla u^*_h + \kap^{-1/2} \mbf{q}_h \|_{T}^2  \|\kap^{1/2} \nabla u^*_h - \kap^{1/2} \nabla \wtil{u}^*_h \|_{T}^2 \right)  \\
	\label{eq:posteriori_T1}
	& + \dfrac{\what{C}_1}{4\, \epsilon_1} \sum_{T\in \mc{T}_h} h_T^2 \| \mc{F}(u^*_h) - P_W \mc{F}(u^*_h) \|^2_{T},
	\end{align}
where $\what{C}_1:= \max\{1, 3\, c\, \underline{\kap}^{-1}\}$.
%
\paragraph{Bound for $\md{T}_{2}$.}
%
We begin by rewriting  $\md{T}_2$ as 
	 \begin{align*}
    \md{T}_{2}
	=& \pdual{\mbf{q}_h \cdot \mbf{n}, u - \wtil{u}^*_h}_{\partial \mc{T}_h} + (P_W \mc{F}(u^*_h) - \nabla \cdot \mbf{q}_h, (u - \wtil{u}^*_h) - \mc{C}_h(u-\wtil{u}^*_h))_{\mc{T}_h} \\
	& + (P_W \mc{F}(u^*_h) - \nabla \cdot \mbf{q}_h, \mc{C}_h (u-\wtil{u}^*_h))_{\mc{T}_h}, 
	 \end{align*}
where $\mc{C}_h$ is the Cl\'ement interpolation operator defined in Appendix \ref{sec:cle-osw}. Rearranging terms above, using  $u= \wtil{u}^*_h = \varphi$ on $\Gamma_h$, and applying Lemma \ref{lem:prop in W_h}, we have
	\begin{align*}
	\md{T}_{2} &= \pdual{\mbf{q}_h \cdot \mbf{n}, u - \wtil{u}^*_h}_{\partial \mc{T}_h} + (P_W \mc{F}(u^*_h) - \nabla \cdot \mbf{q}_h, (Id_M - \mc{C}_h)(u - \wtil{u}^*_h))_{\mc{T}_h} - \pdual{\mbf{q}_h\cdot \mbf{n}, \mc{C}_h(u-\wtil{u}^*_h)}_{\partial \mc{T}_h} \\
	 &\quad - (\kap \nabla u_h^* + \mbf{q}_h, \nabla \mc{C}_h(u - \wtil{u}^*_h))_{\mc{T}_h}  \\
	 &= \pdual{\mbf{q}_h \cdot \mbf{n}, (Id_M - \mc{C}_h)(u - \wtil{u}^*_h)}_{\partial \mc{T}_h \setminus \Gamma_h} + (P_W \mc{F}(u^*_h) - \nabla \cdot \mbf{q}_h, (Id_M - \mc{C}_h)(u - \wtil{u}^*_h))_{\mc{T}_h} \\
    &\quad  - (\kap \nabla u_h^* + \mbf{q}_h, \nabla \mc{C}_h (u - \wtil{u}^*_h))_{\mc{T}_h}
    \end{align*}
Then, by Young's inequality, 
	\begin{align*}
    |\md{T}_{2}| &\leq \dfrac{1}{4\, \epsilon_2} \sum_{T\in \mc{T}_h} h_T^2 \|P_W \mc{F}(u^*_h) - \nabla \cdot \mbf{q}_h\|^2_{T} + \epsilon_2 \sum_{T\in \mc{T}_h} h_T^{-2} \|(Id_M - \mc{C}_h)(u - \wtil{u}^*_h)\|^2_{T}  \nonumber\\
	&\quad + \dfrac{1}{4\, \epsilon_2} \sum_{e\in \mc{E}_h^\circ} h_e \|\jump{\mbf{q}_h}\|^2_{e} + \epsilon_2 \sum_{e\in \mc{E}_h^i} h_e^{-1} \|(Id_M - \mc{C}_h)(u - \wtil{u}^*_h)\|^2_{e}\nonumber \\
	&\quad + \dfrac{\overline{\kap}}{4\, \epsilon_2} \sum_{T\in \mc{T}_h} \|\kap^{1/2} \nabla u_h^* + \kap^{-1/2}\mbf{q}_h\|^2_{T} + \epsilon_2 \sum_{T\in \mc{T}_h} |\mc{C}_h(u - \wtil{u}^*_h)|^2_{1,T}.\label{auxT3}
	\end{align*}
On the other hand, the properties of  Cl\'ement's interpolant---Lemma \ref{lem. Clement interpolator }---and the Poincar\'e inequality with constant $c_p$ imply that
    \begin{align*}
    \sum_{T\in \mc{T}_h} h_T^{-2} \|(Id_M - \mc{C}_h)(u - \wtil{u}^*_h)\|^2_{T} &\lesssim  \sum_{T\in \mc{T}_h} \|u-\wtil{u}^*_h\|^2_{\Delta_T} \leq  \what{c}_1\, c_p \sum_{T\in \mc{T}_h} |u - \wtil{u}^*_h|^2_{1,T}, \\
     \sum_{e\in \mc{E}_h^\circ} h_e^{-1} \|(Id_M - \mc{C}_h)(u - \wtil{u}^*_h)\|^2_{e} &\lesssim \sum_{e\in \mc{E}_h^\circ}  \|u-\wtil{u}^*_h\|^2_{\Delta_e} \leq  \what{c}_2\, c_p \sum_{T\in \mc{T}_h} |u - \wtil{u}^*_h|^2_{1,T}, \\
     \sum_{T\in \mc{T}_h} |\mc{C}_h(u - \wtil{u}^*_h)|^2_{1,T} &\lesssim \sum_{T\in \mc{T}_h} \|u-\wtil{u}^*_h\|^2_T \leq  \what{c}_3\, c_p \sum_{T\in \mc{T}_h} |u - \wtil{u}^*_h|^2_{1,T}.
    \end{align*}
Above, the sets $\Delta_T$ and $\Delta_e$ correspond to the macro element surrounding the element $T$ and face $e$ respectively, i.e.
\[
\Delta_T:= \{ T'\in \mc{T}_h : \overline{T} \cap \overline{T'} \neq \varnothing \} \qquad  \text{ and } \qquad \Delta_e = \{T'\in \mc{T}_h: \overline{T'} \cap \overline{e} \neq \varnothing \}. 
\]
Then, applying \eqref{eq:est_aux_T1} to the right side terms of the last three inequalities, one arrives at
\begin{align}
\nonumber
|\md{T}_{2}| \leq\,&  \quad \frac{\what{C}_2}{4\epsilon_2}\left( \sum_{T\in \mc{T}_h} h_T^2 \|P_W \mc{F}(u^*_h) - \nabla \cdot \mbf{q}_h\|^2_{T} + \sum_{e\in \mc{E}_h^\circ} h_e \|\jump{\mbf{q}_h}\|^2_{e} \right) \\
\label{eq:posteriori_T2}
&  + \epsilon_2\what{C}_2\sum_{T\in \mc{T}_h}\left( \|\kap^{-1/2} (\mbf{q} - \mbf{q}_h)\|^2_T + \left(\frac{1}{4\, \epsilon_2} + 1\right)  \|\kap^{1/2} \nabla u_h^* + \kap^{-1/2}\mbf{q}_h\|^2_{T} + \|\kap^{1/2}\, \nabla (u-\wtil{u}_h^*)\|_T \right)
\end{align}
with $\what{C}_2= \max\{1, \overline{\kap}, 3\, c_p\, \underline {\kap}^{-1} (\what{c}_1+ \what{c}_2+\what{c}_3) \}$.

%
\paragraph{Bound for $\md{T}_{3}$.}
%
From Young's inequality, it follows that
    \begin{align}\label{eq:posteriori_T3}
    |\md{T}_{3}|
	&\leq \sum_{T\in \mc{T}_h} \left( \dfrac{1}{2\, \epsilon_3} \left(\| \kap^{1/2} \nabla u^*_h  + \kap^{-1/2}\mbf{q}_h \|^2_T
	+
	 \| \kap^{1/2}\, \nabla (\wtil{u}^*_h - u_h^*)\|^2_T \right)
	+ \epsilon_3 \|\kap^{-1/2}(\mbf{q} - \mbf{q}_h)\|^2_T \right).
	\end{align}
%
\paragraph{Bound for $\md{T}_{4}$.}
%
Adding and subtracting $u^*_h$, and using the Lipschitz continuity of $\mc{F}$, we have
	\begin{equation}
	\label{eq:posteriori_T4}
	|\md{T}_{4}| \leq L \sum_{T\in \mc{T}_h} \left(\|u-u^*_h \|^2_{T} + \|u-u^*_h\|_T \|u^*_h- \wtil{u}^*_h\|_T \right)\leq \dfrac{L}{2} \sum_{T\in \mc{T}_h}\left(3\|u - u^*_h\|^2_T + \|u^*_h - \wtil{u}^*_h\|^2_T\right),
	\end{equation}
where the second inequality follows from  Young's inequality.
%
\paragraph{Wrap-up.}
By the decomposition \eqref{q_qh} and the bounds \eqref{eq:posteriori_T1} - \eqref{eq:posteriori_T4} obtained for the terms $\mb{T}_{i}$ , we deduce that
    \begin{align*}
   	&\|\kap^{-1/2}(\mbf{q}-\mbf{q}_h)\|^2_{\compD}
   	\leq \dfrac{\what{C}_1}{4\, \epsilon_1} \sum_{T\in \mc{T}_h} h_T^2 \| \mc{F}(u^*_h) - P_W \mc{F}(u^*_h) \|^2_{0,T} + 	\dfrac{\what{C}_2}{4\, \epsilon_2} \sum_{T\in \mc{T}_h} h_T^2 \|P_W \mc{F}(u^*_h) - \nabla \cdot \mbf{q}_h\|^2_{T} \\
	&\quad  + \left( \what{C}_1\, \epsilon_1 + \what{C}_2\, \epsilon_2 + \dfrac{\what{C}_2}{4\, \epsilon_2} + \dfrac{1}{2\, \epsilon_3 }\right)  \sum_{T\in \mc{T}_h} \|\kap^{1/2} \nabla u_h^* + \kap^{-1/2}\mbf{q}_h\|^2_{T} + \dfrac{\what{C}_2}{4\, \epsilon_2}  \sum_{e\in \mc{E}_h^\circ} h_e \, \|\jump{\mbf{q}_h}\|^2_{e} \\ 
	&\quad + \left( \what{C}_1\, \epsilon_1 + \what{C}_2\, \epsilon_2 + \dfrac{1}{2\, \epsilon_3}\right)\, \overline{\kap}\, 
    \sum_{T\in \mc{T}_h} \| \nabla (u_h^* -  \wtil{u}^*_h)\|^2_{T}  + \dfrac{L}{2} \sum_{T\in \mc{T}_h} \|u^*_h - \wtil{u}^*_h\|^2_T +
    \dfrac{3\, L}{2} \sum_{T\in \mc{T}_h}\|u - u^*_h\|^2_T
	\\
    &\quad  + (\what{C}_1\, \epsilon_1 + \what{C}_2\,  \epsilon_2 + \epsilon_3 )\sum_{T\in \mc{T}_h} \|\kap^{-1/2}(\mbf{q} - \mbf{q}_h)\|^2_T .
	\end{align*}
Finally, considering values of $\epsilon_1, \epsilon_2$, and $\epsilon_3$ such that $\what{C}_1\, \epsilon_1 + \what{C}_2\,  \epsilon_2 + \epsilon_3 < 1/2$, and the estimate for the terms that involve $\wtil{u}_h^*$, given in \eqref{eq:uh-uhtilde},  the proof in concluded with $ \overline{C}_1$ dependent only  of $\what{C}_1$ and $\what{C}_2$.		
\end{proof}
Now, we bound the second term of the error $\aposteerror^2$ (see \eqref{def:error posteriori}).
\begin{lem}\label{lem:uh-uh*}
Under all the previous assumptions, the following bound for the error in the post processed solution holds
	\begin{equation*}
	\begin{array}{rl}
 	\|u-u^*_h\|^2_{\compD} 
	\leq&  \displaystyle   \overline{C}_2 \, \left(   \sum_{T\in \mc{T}_h} \|\kap^{-1/2} \mbf{q}_h + \kap^{1/2}\nabla u^*_h \|^2_T + \sum_{e\in \mc{E}_h^\circ} h_e^{-1} \| \jump{u_h^*} \|^2_e  +  \sum_{e\in \mc{E}_h^{\partial}} h_e^{-1} \| \varphi_h - u_h^*\|^2_e \right. \\[2ex]
	&\qquad + \|\kap^{-1/2}\, (\mbf{q} - \mbf{q}_h) \|^2_{\compD} + \|h_e^{-1/2}\, (\varphi - \varphi_h)\|^2_{\Gamma_h} \Bigg)
	\end{array}
	\end{equation*}
where $ \overline{C}_2>0$ is a positive constant independent of $h$.
\end{lem}
\begin{proof}
First, note that, since $u-\wtil{u}^*_h \in H_0^1(\Omega_h)$, then thanks to the triangle and  Poincar\'e inequalities with constant $c_p$, it follows that 
\[
	\|u-u^*_h\|^2_{\compD} \leq 
2 \,  \| u- \wtil{u}^*_h \|^2_{\compD} + 2\,  \|\wtil{u}^*_h - u^*_h\|^2_{\compD} \leq 2\,c_p^2 \|\nabla u - \nabla \wtil{u}^*_h\|^2_{\Omega_h} +  2 \, \|\wtil{u}^*_h - u^*_h\|^2_{\compD}.
\]
then, since $\mbf{q}= -\kap^{-1}\, \nabla u$ (see \eqref{eq:mixed_a}) adding $\pm \kap^{-1} \mbf{q}_h$, we get
	\begin{align*}
 	\|u-u^*_h\|^2_{\compD} 
	&\leq  4\, c_p^2\,  \underline{\kap}^{-1}  \left( \|\kap^{-1/2} (\mbf{q} - \mbf{q}_h) \|^2_{\compD} +  \|\kap^{-1/2} \mbf{q}_h + \kap^{1/2}\nabla\wtil{u}^*_h \|^2_{\compD} \right) + 2\,   \|\wtil{u}^*_h-u^*_h\|^2_{\compD}.
	\end{align*}
Now, adding $\pm \kap^{-1/2}\, \nabla u_h^*$, results in
	\begin{equation}\label{eq:u-uh* aux}
	\begin{array}{rl}
 	\|u-u^*_h\|^2_{\compD} 
	\leq\,&  4\, c_p^2\, \underline{\kap}^{-1}\,  \|\kap^{-1/2} (\mbf{q} - \mbf{q}_h) \|^2_{\compD} + 8\, c_p^2\, \underline{\kap}^{-1}\,  \|\kap^{-1/2} \mbf{q}_h
	+ \kap^{1/2}\nabla u^*_h \|^2_{\compD}  \\
	& \qquad +2\, \max\{4\, c_p^2\, \underline{\kap}^{-1}\, \overline{\kap}, 1\} \left(   \|\nabla (u_h^*- \wtil{u}^*_h) \|^2_{\compD} +  \, \|\wtil{u}^*_h-u^*_h\|^2_{\compD}\right).
	\end{array}
	\end{equation}
Finally, the proof is concluded by substituting \eqref{eq:uh-uhtilde} into \eqref{eq:u-uh* aux}. 
\end{proof}
We conclude this part with an estimate for the last term of our error,
\begin{lem}\label{lem:varphi-varphih}
Assume that all previous assumptions are satisfied. Then, there exists a positive constant $ \overline{C}_3$, independent of $h$ such that
    \begin{equation*}
    \begin{array}{rl}
    \|h_e^{-1/2}\, (\varphi-\varphi_h)\|^2_{\Gamma_h} \leq& \displaystyle  \overline{C}_3\,   \max_{e\in \mc{E}_h^{\partial}} \{ r_e^2, r_e\, (C_{ext}^e)^2\}\, \Bigg( \|\kap^{-1/2} (\mbf{q} - \mbf{q}_h)\|^2_{\compD} + h^2\, L^2\,  \|u-u_h^*\|_{\compD}^2 \\[2ex]
    &\qquad \displaystyle  + \sum_{T\in \mc{T}_h}  h_T^2\, \|P_W \mc{F}(u_h^*)-\nabla\cdot  \mbf{q}_h\|^2_T + \text{osc}^2(\mc{F}, \mc{T}_h) \Bigg).
    \end{array}
    \end{equation*}
\end{lem}
\begin{proof}
We first notice that this term depends on what happens in the domain $\compD^c$, that is
    \begin{align}\label{eq:varphi-varphih aux}
    \|h_e^{-1/2}\, (\varphi - \varphi_h)\|^2_{\Gamma_h} \lesssim \|\kap^{-1/2}(\mbf{q}-\mbf{q}_h)\|^2_{\compD^c}.
    \end{align}
Then, for each $T \in \mc{T}_h$, we have
    \begin{equation}\label{eq:div q-qh}
    \begin{array}{rl}
    h_T\|\nabla \cdot (\mbf{q}- \mbf{q}_h)\|_{T} =&  h_T\|\mc{F}(u)-\nabla\cdot  \mbf{q}_h\|_{T}  \\[2ex]
    \leq& h_T\|P_W \mc{F}(u_h^*)-\nabla\cdot  \mbf{q}_h\|_{T} + h_T\|\mc{F}(u)-\mc{F}(u_h^*)\|_{T} + h_T \|P_W \mc{F}(u_h^*) - \mc{F}(u_h^*) \|_T .
    \end{array}
    \end{equation}
Now we will need to consider the approximation error measured in a function space with additional regularity. For $T\in \mc{T}_h$ let $E_T: \mbf{H}(\textbf{div};T) \to \mbf{H}(\textbf{div};\md{R}^d)$ be any local extension operator, and $Q_k(E_T(\mbf{q}))\in\mathbb P_k(\mathbb R^d)$ the averaged averaged Taylor polynomial of degree $k$ introduced in the proof of Lemma \ref{lem:prop in W_h}. Let also  $E: \mbf{H}(\textbf{div};\mc{T}_h) \to \mbf{H}(\textbf{div};\md{R}^d)$ be a global extension such that $E(\mbf{v})\vert_T := E_T(\mbf{v})$ for all $T\in \mc{T}_h$ and $\mbf{v}\in \mbf{H}(\textbf{div};\mc{T}_h)$. Note that 
    \begin{align*}
    \|\mbf{q}-\mbf{q}_h\|_{\compD^c} &\leq \|Q_k(E(\mbf{q}))-\boldsymbol q_h\|_{\compD^c} + \|\mbf{q}-Q_k(E(\mbf{q}))\|_{\compD^c} \\
     &= \|\mbf{q}-Q_k(E(\mbf{q}))\|_{\compD^c} + \left( \sum_{e\in \mc{E}_h^{\partial}} \|Q_k(E(\mbf{q}))-\mbf{q}_h\|^2_{T^e_{ext}}  \right)^{1/2} \\
    &\leq  \|E(\mbf{q}) - Q_k(E(\mbf{q}))\|_{\compD^c} + \left( \sum_{e\in \mc{E}_h^{\partial}} r_e\, (C_{ext}^e)^2 \|Q_k(E(\mbf{q}))-\mbf{q}_h\|^2_{T^e}  \right)^{1/2}.
    \end{align*}
Since $\|Q_k(E(\mbf{q}))-\mbf{q}_h\|^2_{T^e}=\|Q_k(E(\mbf{q})-\mbf{q}_h)\|^2_{T^e}\lesssim\|E(\mbf{q})-\mbf{q}_h\|^2_{T^e}=\|\mbf{q}-\mbf{q}_h\|^2_{T^e} $, we obtain
    \begin{align*}
    \|\mbf{q}-\mbf{q}_h\|_{\compD^c} &\lesssim\|E(\mbf{q}) - Q_k(E(\mbf{q}))\|_{\compD^c} +   \max_{e\in \mc{E}_h^{\partial}} \{r_e^{1/2}\, C_{ext}^e\}  \|\mbf{q}-\mbf{q}_h\|_{\compD}. 
    \end{align*}
Moreover, to bound the first term on the right hand side, we observe that
  \begin{align*} 
  \|E_{T^e}(\mbf{q}) - Q_k(E_{T^e}(\mbf{q}))\|^2_{T^e_{ext}}
  &= \|E_{T^e}(\mbf{q}-\mbf{q}_h) - Q_k(E_{T^e}(\mbf{q})-E_{T^e}(\mbf{q_h}))\|_{T^e_{ext}}^2\\
  &\lesssim \|E_{T^e}(\mbf{q}-\mbf{q}_h)\|_{T^e_{ext}}^2 \lesssim  r_e^2 \, \|\kap^{-1/2}( \mbf{q}-\mbf{q}_h) \|_{T^e} + r_e^2\, h_T^2\, \|\nabla\cdot (\mbf{q} - \mbf{q}_h)\|^2_{T^e},
   \end{align*}
where we have used the estimate in \eqref{eq: estim_operator E a}. Thus,
    \begin{align*}
 \|\kap^{-1/2}\, (\mbf{q} - \mbf{q}_h)\|_{\compD^c}
    &\lesssim \displaystyle  \max_{e\in \mc{E}_h^{\partial}} \{ r_e^2, r_e\, (C_{ext}^e)^2\}\, \left( \|\kap^{-1/2} (\mbf{q} - \mbf{q}_h)\|^2_{\compD} + h_T^2 \|\nabla \cdot (\mbf{q} - \mbf{q}_h)\|_{\compD}^2 \right).
    \end{align*}
The result follows combining the last inequality with \eqref{eq:varphi-varphih aux}.
\end{proof}
With all the pieces in place, we can now show that the error in the flux can be successfully estimated if one considers the data oscillation.

\begin{thm}\label{thm:reliable}
Assume that the hypotheses of Lemmas \ref{lem:q-qh}--\ref{lem:varphi-varphih} hold. In addition, if 
    \begin{subequations}\label{asumptions}
    \begin{align}
     \overline{C}_1\,  \overline{C}_3\, \max_{e\in \mc{E}_h^{\partial}} \{ r_e^2, r_e\, (C_{ext}^e)^2\}   &< 1/2 ,\label{assumption1 thm_reliable} \\[2ex]
     \overline{C}_2\, L\, (L\, h^2 + 2\,  \overline{C}_1) &< 1/2, \label{assumption2 thm_reliable} \\[2ex]
    (2\,  \overline{C}_2+1)\,   \overline{C}_3\, \max_{e\in \mc{E}_h^{\partial}} \{ r_e^2, r_e\, (C_{ext}^e)^2\} &< 1/2 \label{assumption3 thm_reliable},
    \end{align}
    \end{subequations}
where $ \overline{C}_i, \ i\in \{1,2,3\}$ are defined in the Lemmas \ref{lem:q-qh}, \ref{lem:uh-uh*} and \ref{lem:varphi-varphih}. Then, there exists a positive constant $C_{rel}$, such that 
    \begin{equation*}
   \aposteerror^2 \leq C_{rel} \left( \eta^2 +  \text{osc}^2(\mc{F},\mc{T}_h) \right).
   	\end{equation*}
\end{thm}
\begin{proof}
We first replace the estimation of the Lemma \ref{lem:varphi-varphih} into Lemma \ref{lem:q-qh} and, together with assumption \eqref{assumption1 thm_reliable}, obtain
    \begin{equation}\label{eq:q-qh 2}
    \|\kap^{-1/2}\, (\mbf{q} - \mbf{q}_h) \|^2_{\compD} \leq L\, (L\, h^2 + 2\,  \overline{C}_1)\, \|u-u_h^*\|^2_{\compD} +  (2\,  \overline{C}_1+1) \left( \text{osc}^2(\mc{F},\mc{T}_h) +	\eta^2 \right).
    \end{equation}
Combining  the assumption \eqref{assumption2 thm_reliable} with \eqref{eq:q-qh 2} into the Lemma \ref{lem:uh-uh*}, we obtain
    \begin{equation}\label{eq:u-uh* 2}
    \|u-u_h^*\|^2_{\compD} \leq 2\, \overline{C}_2\, \|h_e^{-1/2}\, (\varphi-\varphi_h)\|^2_{\Gamma_h} + 2\, ( \overline{C}_1 + 1)\, \overline{C}_2 \, \left( \text{osc}^2(\mc{F},\mc{T}_h) +	\eta^2 \right).
    \end{equation}
Note that, thanks to \eqref{eq:u-uh* 2} and assumption \eqref{assumption2 thm_reliable}, the estimation \eqref{eq:q-qh 2} can be rewritten as 
    \begin{equation}\label{eq:q-qh 3}
    \|\kap^{-1/2}\, (\mbf{q} - \mbf{q}_h) \|^2_{\compD} \leq \|h_e^{-1/2}\, (\varphi-\varphi_h)\|^2_{\Gamma_h} + (3\, \overline{C}_1 + 2)\, \left( \text{osc}^2(\mc{F},\mc{T}_h) +	\eta^2 \right).
    \end{equation}
Combining \eqref{eq:u-uh* 2} with \eqref{eq:q-qh 3} and using the Lemma \ref{lem:varphi-varphih} again, we arrives at
    \begin{align*}
     &\|\kap^{-1/2}\, (\mbf{q} - \mbf{q}_h) \|^2_{\compD}  + \|u-u_h^*\|^2_{\compD} \\
     &\quad \leq (2\,  \overline{C}_2+1)\,   \overline{C}_3\, \max_{e\in \mc{E}_h^{\partial}} \{ r_e^2, r_e\, (C_{ext}^e)^2\} \, \left( \|\kap^{-1/2}\, (\mbf{q} - \mbf{q}_h) \|^2_{\compD}  + \|u-u_h^*\|^2_{\compD}  \right) + \what{c}\,\left( \text{osc}^2(\mc{F},\mc{T}_h) +	\eta^2 \right).
    \end{align*}
Then, by assumption \eqref{assumption3 thm_reliable}, we deduce
    \begin{equation*}
    \|\kap^{-1/2}\, (\mbf{q} - \mbf{q}_h) \|^2_{\compD}  + \|u-u_h^*\|^2_{\compD} \lesssim \text{osc}^2(\mc{F},\mc{T}_h) +	\eta^2. 
    \end{equation*}
Finally, observe that the above estimation allows us rewritten the Lemma \ref{lem:varphi-varphih} as
    \begin{equation*}
    \|h_e^{-1/2} \, (\varphi - \varphi_h)\|^2_{\Gamma_h} \lesssim \text{osc}^2(\mc{F},\mc{T}_h) +	\eta^2. 
    \end{equation*}
which concludes the proof.
\end{proof}

Having established the reliability of the estimator we can now adapt arguments from the linear case to show that the estimator is locally efficient as well. This will follow readily from the following estiamtes.

\begin{thm}\label{thm:efficiency}
Suppose that $L\, h\leq 1$. Then we can assert the following local estimates
    \begin{align*}
    \|\kap^{-1/2}\, \mbf{q}_h + \kap^{1/2}\, \nabla u_h^*\|_T &\lesssim \|\kap^{-1/2}\, (\mbf{q} - \mbf{q}_h) \|_T, \\
    h_e^{-1} \|\jump{u_h^*}\|_e^2 &\lesssim \!\!\!\sum_{T\in \mc{U}_h(e)} \!\!\!\ \|\kap^{-1/2} ( \mbf{q} - \mbf{q}_h)\|^2_T \quad \forall \, e\in \mc{E}_h^{\circ}, \\
    h_e^{-1} \, \|\varphi_h - u_h^*\|_e &\lesssim \displaystyle  \sum_{T\in \mc{U}_h(e)}  \|\kap^{-1/2}(\mbf{q} - \mbf{q}_h)\|_T^2 + \| h_e^{-1/2}\, (\varphi - \varphi_h)\|^2_{e} \quad \forall \, e\in \mc{E}_h^{\partial},  \\
    h_e\| \jump{\mbf{q}_h}\|_e^2 &\lesssim \displaystyle \sum_{T\in \mc{U}_h(e)}   \left( \|\kap^{-1/2} ( \mbf{q} - \mbf{q}_h)\|^2_T + h_T^2\|P_W \mc{F}(u_h^*) - \nabla\cdot \mbf{q}_h\|_T^2 \right. \\
    &\qquad + \|u- u_h^*\|^2_T \Bigg) + \text{osc}^2(\mc{F},\mc{U}_h(e)) \qquad \forall \, e\in \mc{E}_h^\circ \\
    h_T^2\, \|P_W \mc{F}(u^*_h) - \nabla \cdot \mbf{q}_h\|_T  &\lesssim \|\kap^{-1/2} ( \mbf{q} - \mbf{q}_h)\|^2_T + h_T^2 \, \|P_W \mc{F}(u_h^*) - \mc{F}(u_h^*)\|^2_T + \|u-u_h^*\|_T^2 .
    \end{align*}
\end{thm}
\begin{proof}
Note that due to the presence of the non-linear source term, the post-processing defining $u_h^*$ is also non linear, and a direct application of the results in \cite[Lemmas 4.4--4.5]{aposteriori2} and \cite[Lemmas 3.4-3.7]{CoZh2014} is not possible. We then proceed to adapt those arguments to the current semi-linear case and treat each of the estimates above separately in what follows. Local efficiency will follow by combining these estimates. 

\paragraph{Bound for $\|\kap^{-1/2}\, \mbf{q}_h + \kap^{1/2}\, \nabla u_h^*\|_T$.}
This term can be bounded by an application  of \cite[Lemma 3.7]{CoZh2014}, that is
    \begin{align}\label{eq:efficiency 1}
    \|\kap^{-1/2}\, \mbf{q}_h + \kap^{1/2}\, \nabla u_h^*\|_T \lesssim \|\kap^{-1/2}\, (\mbf{q} - \mbf{q}_h) \|_T
    \end{align}
%
\paragraph{Bound for $h_e^{-1} \|\jump{u_h^*}\|^2_e$.} 
%
We begin by splitting $\jump{u_h^*}$ into its component in the space $M_0:= \{ \mu\in L^2(\partial \mc{T}_h): \mu|_e\in {\md{P}_0(e)}, \ \forall \,  e\in \mc{E}_h  \}$ and its orthogonal complement.  Considering $P_{M_0}$, the $L^2(\Omega)-$orthogonal projection into $M_0$,  and $Id_M$ the identity operator on the same space we have 
    \begin{align}\label{eq:efficiency jump uh** 1}
    h_e^{-1} \, \|\jump{u_h^*}\|_e^2 \lesssim h_e^{-1}\, \|P_{M_0}\jump{u_h^*}\|^2_e + h_e^{-1}\, \|(Id_M - P_{M_0})\jump{u_h^*}\|^2_e.
    \end{align}
Each of these terms can be bounded by an application of \cite[Lemma 3.4.  and Lemma 3.5.]{CoZh2014} to all the interior faces of the triangulations. That is, for each $e\in \mc{E}_h^\circ$,
    \begin{subequations}\label{eq:BoundForTheJump}
    \begin{align}\label{eq:BoundForTheJumpA}
         h_e^{-1}\,  \|P_{M_0}\jump{u_h^*}\|^2_e &\lesssim \displaystyle  \sum_{T\in \mc{U}_h(e)} \|\kap^{-1/2}\, \mbf{q}_h + \kap^{1/2}\, \nabla u_h^* \|^2_T, \\
         \label{eq:BoundForTheJumpB}
         h_e^{-1}\, \|(Id_M - P_{M_0})\jump{u_h^*}\|^2_e  &\lesssim \displaystyle  \sum_{T\in \mc{U}_h(e)} \|\nabla (u - u_h^*)\|^2_T.
    \end{align}
    \end{subequations}
Now, adding and subtracting  $\kap^{-1/2} \mbf{q}_h$ to $\nabla (u - u_h^*)$ and using the definition of the flux it follows that
    \begin{equation}\label{eq:nabla u-uh*}
    \|\nabla(u-u_h^*)\|^2_T \lesssim \|\kap^{-1/2}\, (\mbf{q} - \mbf{q}_h) \|^2_T +  \|\kap^{-1/2}\, \mbf{q}_h + \kap^{1/2}\, \nabla u_h^* \|^2_T.
    \end{equation}
Moreover, using the fact that $ \|\kap^{-1/2}\, \mbf{q}_h + \kap^{1/2}\, \nabla u_h^* \|^2_T \lesssim \|\kap^{-1/2}\, (\mbf{q} - \mbf{q}_h) \|^2_T$ (see \eqref{eq:efficiency 1}), we can bound the second term above. The same argument can be applied to \eqref{eq:BoundForTheJumpA}, and combining these two results we arrive at
    \begin{align}\label{eq:efficiency jump uh*}
     h_e^{-1} \, \|\jump{u_h^*}\|_e^2 \lesssim \sum_{T\in \mc{U}_h(e)}  \|\kap^{-1/2}\, (\mbf{q} - \mbf{q}_h) \|^2_T \qquad \forall \ e \in \mc{E}_h^{\circ}.
    \end{align}

\paragraph{Bound for $h_e^{-1} \|\varphi_h - u_h^*\|_e^2$.} 
%
First, we define for each $T\in \mc{T}_h$, the local Raviart-Thomas \cite{Gatica:2014} space of order $k$ as
    \begin{equation*}
    \mb{RT}_k(T) \,:=\, [\md{P}_k(T)]^d \oplus \md{P}_k(T)\,\mbf{x},
    \end{equation*}
where  $\md{P}_k(T)$ denotes the space of polynomials of degree at most $k$ defined in $T\in \mc{T}_h$ (see Section \ref{sec:HDG}).

Taking as test in \eqref{eq:HDG_a} $\mbf{v} \in \mb{RT}_0(T)$, it is possible to use the second equation defining the post processing $u_h^*$, \eqref{eq:post_processing2}, for $\nabla \cdot \mbf{v}$ belongs to the space of piece-wise constant functions $\md{P}_0(T)$, to obtain
    \begin{equation*}
    (\kap^{-1}\, \mbf{q}_h, \mbf{v} )_T - (u_h^*, \nabla \cdot \mbf{v})_T + \langle \what{u}_h, \mbf{v} \cdot \mbf{n}  \rangle_{\partial T} = (\kap^{-1}\, \mbf{q}_h + \nabla u_h^*, \mbf{v} )_T + \langle \what{u}_h - u_h^*, \mbf{v} \cdot \mbf{n}  \rangle_{\partial T} = 0.
    \end{equation*}
On the other hand, if we consider $\mbf{v}\in \bH(\textrm{div}, \mc{U}_h(e))$ for each $e\in \mc{E}_h^{\partial}$, then by summing over all $T\in \mc{U}_h(e),$ we arrive at 
    \begin{equation*}
    \sum_{T\in \mc{U}_h(e)} (\kap^{-1}\ \mbf{q}_h + \nabla u_h^*, \mbf{v})_T + \sum_{T\in \mc{U}_h(e)} \sum_{F\in \partial T \setminus e} \langle \what{u}_h - u_h^*, \mbf{v} \cdot \mbf{n}  \rangle_{\partial T} - \langle \varphi_h - u_h^*, \mbf{v}\cdot\boldsymbol n  \rangle_e.
    \end{equation*}
Since  $\mbf{v}\in \bH(\textrm{div}, \mc{U}_h(e))$ is arbitrary, we can choose it such that, on each $T\in \mc{U}_h(e)$, belongs to $\mb{RT}_0(T)$ and satisfies
    \begin{align*}
    \int_e \mbf{v} \cdot \mbf{n} &= \int_e P_{M_0}(\varphi_h- u_h^*)\cdot \mbf{n} \quad &&\text{for the face } e, \\
    \int_F \mbf{v} \cdot \mbf{n} &= 0 \quad &&\forall \, F\in \partial T \setminus e,
    \end{align*}
we obtain
    \begin{align*}
    \|P_{M_0}(\varphi_h- u_h^*)\|_e^2 = \sum_{T\in \mc{U}_h(e)} (\kap^{-1}\ \mbf{q}_h + \nabla u_h^*, \mbf{v})_T.
    \end{align*}
Then, from the Cauchy--Schwarz inequality and a standard scaling argument $\|\mbf{v}\|_T \lesssim h_e^{1/2}\, \|\mbf{v}\cdot\boldsymbol n\|_e$, we get 
    \begin{equation}\label{eq:vaprhi_h-uh* 1}
    h_e^{-1} \, \|P_{M_0}(\varphi_h - u_h^*)\|_e^2 \lesssim \sum_{T\in \mc{U}_h(e)} \|\kap^{-1/2} \mbf{q}_h + \kap^{1/2}\, \nabla u_h^* \|_T^2
    \end{equation}
Now, analogously to \cite[Lema 3.5]{CoZh2014}, it follows that for each boundary face $e\in \Gamma_h$,
    \begin{align*}
    h_e^{-1} \, \|(Id_M - P_{M_0})(\varphi_h-u_h^*)\|_e^2 &=  h_e^{-1} \, \|(Id_M - P_{M_0})(\varphi_h - u_h^*)\|_e^2 \\
    &\lesssim  h_e^{-1} \, \|(Id_M - P_{M_0})(u- u_h^*)\|_e^2 +  h_e^{-1} \, \|(Id_M - P_{M_0})(\varphi - \varphi_h)\|_e^2 \\
    &\lesssim \sum_{T\in \mc{U}_h(e)} \|\nabla (u-u_h^*)\|_T^2 + h_e^{-1}\, \|\varphi - \varphi_h\|_e^2.
    \end{align*}
Therefore, by \eqref{eq:nabla u-uh*} it follows that
    \begin{equation}\label{eq:varphi_h-uh* 2}
    \begin{array}{rl}
    h_e^{-1} \, \|(Id_M - P_{M_0})(\varphi_h-u_h^*)\|_e^2 &\lesssim \displaystyle\sum_{T\in \mc{U}_h(e)} ( \|\kap^{-1/2}(\mbf{q} - \mbf{q}_h)\|_T^2 + \|\kap^{-1/2} \, \mbf{q}_h + \kap^{1/2}\, \nabla u_h^*\|_T^2) \\[2ex]
    &\quad + h_e^{-1}\, \|\varphi - \varphi_h\|_e^2.
    \end{array}
    \end{equation}
Finally, we decompose
\[
\varphi_h - u_h^* = P_{M_0}(\varphi_h - u_h^*) + (Id_M - P_{M_0})(\varphi_h - u_h^*)
\]
and apply \eqref{eq:vaprhi_h-uh* 1} and \eqref{eq:varphi_h-uh* 2}, to arrive at
    \begin{align}
    \nonumber
    h_e^{-1} \, \|\varphi_h - u_h^*\|_e &\lesssim \, \sum_{T\in \mc{U}_h(e)} ( \|\kap^{-1/2}(\mbf{q} - \mbf{q}_h)\|_T^2 + \|\kap^{-1/2} \, \mbf{q}_h + \kap^{1/2}\, \nabla u_h^*\|_T^2)  + h_e^{-1}\, \|\varphi - \varphi_h\|_e^2 \\
    \label{eq:varphih-uh*}
    &\lesssim \,  \sum_{T\in \mc{U}_h(e)}  \|\kap^{-1/2}(\mbf{q} - \mbf{q}_h)\|_T^2 +  \|h_e^{-1/2}\,(\varphi - \varphi_h)\|^2_{e},
    \end{align}
where we have applied \eqref{eq:efficiency 1} in the second line. 
%
\paragraph{Bound for $h_e \|\jump{\mbf{q}_h}\|_e^2$.}
%
For the interior faces, we have that for any $w\in H_0^1(\mc{U}_h(e))$, then
    \begin{align*}
    \langle\jump{\mbf{q}_h}, w\rangle_e &= \sum_{T\in \mc{U}_h(e)} \langle  (\mbf{q} - \mbf{q}_h) \cdot \mbf{n}, w \rangle_{\partial T} =  \sum_{T\in \mc{U}_h(e)}  \Big( (\mbf{q} - \mbf{q}_h) , \nabla w)_T + (\mc{F}(u) - \nabla \cdot \mbf{q}_h, w)_T \Big) \\
    &\leq  \sum_{T\in \mc{U}_h(e)} \Big( \overline{\kap}^{1/2}\,  \|\kap^{-1/2} ( \mbf{q} - \mbf{q}_h)\|_T \, \|\nabla w\|_T + h_T\, \|\mc{F}(u) - \nabla \cdot \mbf{q}_h\|_T \, h_T^{-1}\,  \|w\|_T   \Big) \\
    &\leq  \sum_{T\in \mc{U}_h(e)} \Big( \overline{\kap}^{1/2}\,  \|\kap^{-1/2} ( \mbf{q} - \mbf{q}_h)\|_T + h_T\, \|\mc{F}(u) - \nabla \cdot \mbf{q}_h\|_T \Big) \Big( \|\nabla w\|_T +  \, h_T^{-1}\,  \|w\|_T   \Big).
    \end{align*}
By choosing a test function of the form $w=B_e\jump{\mbf{q}_h}\in \md{P}_{k+d}(T)$, whith being $B_e$ is a face bubble function defined in \eqref{ineq:bubble functions}, it follows that
    \begin{equation*}
    \int_e B_e \jump{\mbf{q}_h}^2 \lesssim  \sum_{T\in \mc{U}_h(e)}  \Big(  \|\kap^{-1/2} ( \mbf{q} - \mbf{q}_h)\|_T + h_T\, \|\mc{F}(u) - \nabla \cdot \mbf{q}_h\|_T \Big) h_T^{-1}\, h_e^{1/2} \|B_e\jump{\mbf{q}_h}\|_e,
    \end{equation*}
then, due to $h_T^{-1} \, h_e^{1/2} \lesssim h_e^{-1/2}$ and the fact that
    \begin{equation*}
        \int_e B_e^2 \jump{\mbf{q}_h}^2 \lesssim \int_e \jump{\mbf{q}_h}^2 \lesssim \int_e B_e \jump{\mbf{q}_h}^2
    \end{equation*}
one arrives at 
    \begin{equation*}
    h_e\,\| \jump{\mbf{q}_h}\|_e^2 \lesssim \sum_{T\in \mc{U}_h(e)}   \left( \|\kap^{-1/2} ( \mbf{q} - \mbf{q}_h)\|^2_T + h_T^2\|\mc{F}(u) - \nabla\cdot \mbf{q}_h\|_T^2 \right).
    \end{equation*}
Now, using \eqref{eq:div q-qh} and the Lipschitz continuity of $\mc{F}$, due to $L\, h <1$ we get
    \begin{equation}\label{eq:jump qh 1}
    \begin{array}{rl}
    h_e\| \jump{\mbf{q}_h}\|_e^2 &\lesssim \displaystyle \sum_{T\in \mc{U}_h(e)}   \left( \|\kap^{-1/2} ( \mbf{q} - \mbf{q}_h)\|^2_T + h_T^2\|P_W \mc{F}(u_h^*) - \nabla\cdot \mbf{q}_h\|_T^2 +  \|u-u_h^*\|^2_T \right) \\
    &\qquad +\text{osc}^2(\mc{F},\mc{U}_h(e)) .
    \end{array}
    \end{equation}
%

%
\paragraph{Bound for $h_T^2 \|P_W \mc{F}(u^*_h) - \nabla \cdot \mbf{q}_h\|_T^2$.}
%
For each element $T\in \mc{T}_h$ and any function $w\in H_0^1(T)$, we have that
    \begin{align*}
    (P_W\mc{F}(u_h^*) - \nabla \cdot \mbf{q}_h, w)_T &= (P_W\mc{F}(u_h^*) - \mc{F}(u_h^*),w )_T + (\mc{F}(u_h^*) - \mc{F}(u),w )_T + (\mc{F}(u) - \nabla \cdot \mbf{q}_h, w)_T \\
    &= (P_W\mc{F}(u_h^*) - \mc{F}(u_h^*),w )_T + (\mc{F}(u_h^*) - \mc{F}(u),w )_T -(\mbf{q}-\mbf{q}_h,\nabla w)_T . 
    \end{align*}
We now consider the element bubble function  $B_T$ defined in Lemma \ref{lem::bubble_function}
and take $w:= B_Tv,$ with $v:=P_W \mc{F}(u_h^*) -\nabla \cdot \mbf{q}_h$. Then, the equation above yields
    \begin{align*}
     (v, B_T v)_T  \lesssim  \Big( h_T^{-1} \, \|\kap^{-1} \, (\mbf{q} - \mbf{q}_h)\|_T + \|P_W\mc{F}(u_h^*) - \mc{F}(u_h^*)\|_T +  L\, \|u-u_h^*\|_T \Big) \Big( h_T\, \|\nabla (B_Tv)\|_T + \|B_Tv\|_T\Big) .
    \end{align*}
Then, due to  \eqref{ineq:bubble functions} and the inverse inequality $h_T \, \|\nabla w\|_T + \|w\|_T \lesssim \|w\|_T$, we obtain
    \begin{align*}
     \|v\|_T^2  \lesssim  \Big( h_T^{-1} \, \|\kap^{-1} \, (\mbf{q} - \mbf{q}_h)\|_T + \|P_W\mc{F}(u_h^*) - \mc{F}(u_h^*)\|_T +  L\, \|u-u_h^*\|_T \Big) \|B_Tv\|_T.
    \end{align*}
Since $\|B_Tv\|_T\lesssim \|v\|_T$ also by \eqref{ineq:bubble functions}, we have
    \begin{align*}
     \|v\|_T  \lesssim   h_T^{-1} \, \|\kap^{-1} \, (\mbf{q} - \mbf{q}_h)\|_T + \|P_W\mc{F}(u_h^*) - \mc{F}(u_h^*)\|_T +  L\, \|u-u_h^*\|_T.
    \end{align*}
Equivalently, since $Lh <1$, the estimate above can be rewritten as
    \begin{align}\label{eq:efficiency F-div qh}
    h_T^2\, \|P_W \mc{F}(u^*_h) - \nabla \cdot \mbf{q}_h\|_T  \lesssim  \|\kap^{-1/2} ( \mbf{q} - \mbf{q}_h)\|^2_T + h_T^2 \, \|P_W \mc{F}(u_h^*) - \mc{F}(u_h^*)\|^2_T + \|u-u_h^*\|_T^2 .
    \end{align}
%
\end{proof}

Putting together the bounds from the previous theorem and the definition of the local error estimator, the local efficiency follows as an easy corollary.
%
\section*{Acknowledgments}
 Nestor S\'anchez is supported by the Scholarship Program of CONICYT-Chile. Tonatiuh S\'anchez-Vizuet was partially funded by the US Department of Energy. Grant No. DE-FG02-86ER53233. Manuel E. Solano was partially funded by CONICYT--Chile through FONDECYT project No. 1200569 and by Project AFB170001 of the PIA Program: Concurso Apoyo a Centros Cient\'ificos y Tecnol\'ogicos de Excelencia con Financiamiento Basal.

\newpage

\appendix
\setcounter{lem}{0}
\renewcommand{\thelem}{\Alph{section}\arabic{lem}}
\section{HDG projection}\label{sec:HDGprojection}	
%
In order to make this manuscript self-contained, in this section we provide previous results that will help us to analyze our discrete scheme. First of all we recall the HDG projection operators introduced by \cite{CoGoSa2010}. Given constants $l_u, l_{\mbf{q}} \in [0,k]$ and  a pair of functions $(\boldsymbol q,u) \in H^{1+l_q}(T) \times H^{1+l_u}(T)$, we denote by $\bsy{\Pi}(\mbf{q},u):=(\bsy{\Pi}_{\mathrm v}\mbf{q},\Pi_{\mathrm w} u)$ the projection over $\mbf{V}_h\times W_h$ defined as the unique element-wise solutions of
	\begin{subequations}\label{eq:HDGprojector}
	\begin{align}
	(\bsy{\Pi}_{\mathrm v}\mbf{q}, \mbf{v})_T &= (\mbf{q}, \mbf{v})_T &  &\forall \ \mbf{v} \in [\md{P}_{k-1}(T)]^d, \label{properties projector Pi_v} \\
	(\Pi_{\mathrm w} u, w)_T &= (u,w)_T & &\forall \ w\in \md{P}_{k-1}(T), \label{properties projector Pi_w} \\
	\pdual{\bsy{\Pi}_{\mathrm v}\mbf{q}\cdot \mbf{n} + \tau \Pi_{\mathrm w} u, \mu}_{F} &= \pdual{\mbf{q} \cdot \mbf{n} + \tau u, \mu}_F & &\forall \ \mu \in \md{P}_k(F), \label{properties projector Pi_h}
	\end{align}
    \end{subequations}
for every element $T\in \mc{T}_h$, and $F\in \partial T$. The $L^2$ projection into $M_h$ will be denoted as $P_M$. If the stabilization function is chosen so that $\tau_T^{\max} := \max \tau|_{\partial T}>0$, then by \cite{CoGoSa2010} there is a constant $C>0$ independent of $T$ and $\tau$ such that   
    \begin{subequations}\label{eq:projection_error}
    \begin{align}
	\|\bsy{\Pi}_{\mathrm v}\mbf{q} - \mbf{q}\|_T &\leq C h_T^{l_{\mbf{q}}+1} |\mbf{q}|_{\mbf{H}^{l_{\mbf{q}}+1}(T)} +  C h_T^{l_u+1} \tau_T^* |u|_{H^{l_u+1}(T)}, \label{error_projector1}\\
	\|\Pi_{\mathrm w} u - u\|_T &\leq C h_T^{l_u+1} |u|_{H^{l_u+1}(T)} + C \dfrac{h_T^{l_{\mbf{q}}+1}}{\tau_T^{\max}} |\nabla \cdot \mbf{q}|_{H^{l_{\mbf{q}}}(T)}\label{error_projector2}. 
	\end{align}
    \end{subequations}
Here $\tau_T^* := \max \tau|_{\partial T \setminus F^*}$ and $F^*$ is a face of $T$ at which $\tau|_{\partial T}$ is maximum. As is customary, the symbol $|\cdot|_{H^s}$ is to be understood as the Sobolev semi norm of order $s\in\mathbb R$.
%

%
\section{Proof of Lemma \ref{lem:estim post-proc}}\label{sec:AppendixB}
%
In this section we present the proof of Lemma \ref{lem:estim post-proc}, relating to the well posedness of the auxiliary non linear local problem that leads to the post processed approximation $u^*_h$. We re state the Lemma here for convenience.

\paragraph{Lemma 1.}\textit{
The local post processing $u_h^*$ is well defined for $L$ small enough. Moreover, if $Lh^2<1$ and $k\geq 1$, then
    \begin{subequations}
    \begin{alignat}{6}
      \|u-u_h^*\|_{0,\mc{T}_h} &\lesssim (Rh)^{1/2} (h^{l_u+1} |u|_{l_u+2,\mc{T}_h} + h^{l_u+1} |\mbf{q}|_{l_{\mbf{q}}+2,\mc{T}_h} )+h^{l_u+2} |u|_{l_u+2,\mc{T}_h}+L h^{l_u+1}
  |u|_{l_u+2,\mc{T}_h}, \label{eq:B1a} \\
     |u-u_h^*|_{1,T} &\lesssim h_T^{l_u+1} |u|_{l_u+1,T} + Lh_T\|\varepsilon^u \|_{0,T} +\|\boldsymbol q - \boldsymbol q_h\|_{0,T} + Lh_T\|u-u_h\|_{0,T},      \label{eq:B1b} \\
  \intertext{and}
    \sum_{e\in \mathcal{E}_h^\partial} h_e^{1/2} \|\jump{u_h^*}\|_e &\lesssim \|u-u_h^*\|^{1/2}_{0,\mc{T}_h}   \left( \|u-u_h^*\|_{0,\mc{T}_h}^2 + h^2 |u-u_h^*|^2_{1,\mc{T}_h} \right)^{1/4}.  \label{eq:B1c}
    \end{alignat}
    \end{subequations} 
    } 
\begin{proof}
We will prove first that the problem \eqref{eq:post_processing} is well posed. For this, we will use a fixed point argument. Let $T\in \mc{T}_h$. We define the operator $S: \md{P}_{k+1}(T) \to \md{P}_{k+1}(T)$ as $S (\zeta)= z$, where $z$ is the only solution of 
    \begin{subequations}\label{eq:mixed_point_post_processing}
    \begin{align}
    (\kap \nabla z, \nabla w)_T  &= -(q_h,\nabla w)_T + (\mc{F}(u_h),w)_T  - (\mc{F}(\zeta), w)_T,&& \forall \ w \in \md{P}_{k+1}(T),\\
    (z,w)_T &= (u_h,w)_T , && \forall \ w\in \md{P}_0(T)\label{eq:mixed_point_post_processing2}.
    \end{align}
    \end{subequations}
Note that $S$ is surjective because \eqref{eq:mixed_point_post_processing} is  well-posed. We will show now that $S$ has a unique a fixed point and in that case it is the solution of  \eqref{eq:post_processing}. Let $\zeta_1,\zeta_2 \in \md{P}_{k+1}(T)$ such that $S(\zeta_1)=z_1$ and $S(\zeta_2)=z_2$, with $z_1$ and $z_2$ satisfying \eqref{eq:mixed_point_post_processing}. We observe that $\zeta_1-\zeta_2 \in \md{P}_{k+1}(T)$ and 
    \begin{subequations}\label{eq:z1z2}
    \begin{align}
    (\kap \nabla (z_1-z_2), \nabla w)_T  &= - (\mc{F}(\zeta_1)-\mc{F}(\zeta_2), w)_T,&& \forall \ w \in \md{P}_{k+1}(T),\label{eq:z1z2_1}\\
    (z_1-z_2,w)_T &= 0 , && \forall \ w\in \md{P}_0(T)\label{eq:z1z2_2}.
    \end{align}
    \end{subequations}
Then, for $i=1$ and $2$, we set $\overline{z}_i:= \displaystyle \dfrac{1}{|T|} \int_T z_i$ and noticing that $\overline{z}_1=\overline{z}_2$ by equation \eqref{eq:z1z2_2}, we have 
    \[
    \|z_1-z_2\|^2_{T} = \| (z_1-\overline{z_1}) - (z_2-\overline{z_2}) \|^2_{T}  
  \leq C_F^2 \|\kap^{1/2} \nabla (z_1-z_2)\|^2_T,
    \]
where we have used the Friedrichs inequality with constant $C_F>0$. Taking $w= z_1-z_2$ in \eqref{eq:z1z2_1}, and recalling that $\mathcal F$ is Lipschitz continuous with constant $L$, we obtain
\[
    \|z_1-z_2\|^2_{T} 
    \leq C_F^2 ( \mc{F}(\zeta_2) - \mc{F}(\zeta_1) ,z_1-z_2)_T 
    \leq  C_F^2L \|\zeta_2 - \zeta_1\|_{T} \|z_1-z_2\|_{T}.
\]
Thus, the operator $S$ is a contraction as long as $C_F^2 L < 1$. If that is indeed the case, it has a unique fixed point.

For the inequality \eqref{eq:B1a}, let $P_0$ and $P_{W^*}$ be the $L^2-$projectors into the space of constants and into $W_h^*$ respectively and decompose
    \begin{equation}\label{eq:desc u-u_h*}
    u-u_h^* = (I-P_{W^*})u + P_0(P_{W^*}u - u_h^*) + (I-P_0)(P_{W^*}u-u_h^*) ,
    \end{equation}
We will now proceed to bound each of the terms on the right hand side of this expression separately in order to estimate the difference $u - u^*_h$. For the first term it is easy to see that
    \begin{equation}\label{eq:bound1}
    \| (I - P_{W^*})u\|_{0,T} \lesssim h_T^{l_u+2} |u|_{l_u+2,T}.
    \end{equation}
For the second term we first notice that, since $W^*$ is a space of piecewise polynomials, the definitions of $P_{W^*}$ and $\Pi_{W}$, since $k\geq 1$, imply $ P_0\,P_{W^*}u = P_0u = P_0\,\Pi_{W}u $
    \begin{eqnarray}\label{eq:norm_ident u-u_h*}
    \|P_0(P_{W^* }u - u_h^*)\|_{0,T} =
    \|P_0(\Pi_{W} u - u_h)\|_{0,T} \leq \|\Pi_{W} u - u_h\|_{0,T}= \|\varepsilon^u \|_{0,T}.
    \end{eqnarray}
In the first equality we have made use of the fact that, due the definition of $u_h^*$ in equation \eqref{eq:post_processing2}, we have $P_0 u_h^* = P_0 u_h$. 

Now we move on to the third term in \eqref{eq:desc u-u_h*} and note that for every $\boldsymbol v$ in the space of vector valued functions with components belonging to  $W^*_h$ and $T\in\mathcal T$ it holds that 
\begin{equation}\label{eq:aux1}
(\kap \nabla (u-u_h^*), \boldsymbol v)_T =\, \left(\kap \nabla \left(P_{W^*}u-u_h^*\right),  \boldsymbol v\right)_T = (\kappa\nabla\,(I-P_0)(P_{W^*}u-u_h^*), \boldsymbol v)_T.
\end{equation}
Moreover, for the exact solutions $(u,\boldsymbol q)$, we have $\kappa\nabla u = -\boldsymbol q $ so that the difference $u-u_h^*$ satisfies
\[
    (\kap \nabla (u-u_h^*), \nabla w)_T =\,- (\mbf{q} - \mbf{q}_h, \nabla w)_T + (\mc{F}(u_h^*) - \mc{F}(u),w )_T - (\mc{F}(u_h) - \mc{F}(u),w )_T
\]    
for every $w\in W^*_h$ and $T\in\mathcal T$. Letting $w:=(I-P_0)(P_{W^*}u-u_h^*)\in W^*$ and $\nabla w$ be the test functions above, and using conditions \eqref{eq:aux1} leads to
\[
(\kap \nabla w, \nabla w)_T = -(\mbf{q} - \mbf{q}_h, \nabla w)_T + (\mc{F}(u_h^*) - \mc{F}(u),w )_T + (\mc{F}(u)- \mc{F}(u_h),w)_T.
\]
From this equation, using the the scaling argument $\|w\|_{0,T} \lesssim h_T |w|_{1,T}$ and the inverse inequality $|w|_{1,T}\lesssim h_T^{-1} \|w\|_{0,T}$ we arrive at
\[
    h_T^{-2}\, \|w\|^{2}_{0,T}\lesssim \overline \kappa|w|^2_{1,T}\leq\, \|\boldsymbol q - \boldsymbol q_h\|_{0,T}|w|_{1,T} + L\,  \left(\|u-u^*_h\|_{0,T} + \|u-u_h\|_{0,T}\right)\|w\|_{0,T}
\]   
from which we conclude that
    \begin{equation*}
    \|w\|_{0,T}\lesssim h_T\, \|\boldsymbol q - \boldsymbol q_h\|_{0,T} + L \, h_T^2\,   \left( \|u-u^*_h\|_{0,T} + \|u-u_h\|_{0,T}\right).
    \end{equation*}
Recalling the decomposition \eqref{eq:desc u-u_h*}, and the estimates \eqref{eq:bound1},  \eqref{eq:norm_ident u-u_h*} we can bound the term $\|u-u^*_h\|_{0,T}$ on the right hand side yielding
    \begin{align}\label{eq:bound3}
    (1-Lh^2_T)\|w\|_{0,T}&\lesssim h_T\, \|\boldsymbol q - \boldsymbol q_h\|_{0,T} + L \, h_T^2\,   \left( h_T^{l_u+2} |u|_{l_u+2,T} + \|\varepsilon^u \|_{0,T} + \|u-u_h\|_{0,T}\right)
    \end{align}
Combining \eqref{eq:bound3} above with \eqref{eq:bound1} and  \eqref{eq:norm_ident u-u_h*} once more we arrive at
    \begin{align*}
    (1-\, Lh_T^2)\, \|u-u_h^*\|_{0,T} \lesssim\,& (1-Lh_T^2)h_T^{l_u+2} |u|_{l_u+2,T} + (1-Lh_T^2)\|\varepsilon^u \|_{0,T} +h_T\, \|\boldsymbol q - \boldsymbol q_h\|_{0,T} \\
    &+L \, h_T^2   \left( h_T^{l_u+2} |u|_{l_u+2,T} + \|\varepsilon^u \|_{0,T} + \|u-u_h\|_{0,T}\right)\nonumber\\
    \lesssim\, &h_T^{l_u+2} |u|_{l_u+2,T} + \|\varepsilon^u \|_{0,T} +h_T\, \|\boldsymbol q - \boldsymbol q_h\|_{0,T} +
    L \, h_T^2\,   \|u-u_h\|_{0,T}.
    \end{align*}
So, assuming $Lh_T^2<1$ for each $T\in \mc{T}_h$, results
    \begin{equation*}
     \|u-u_h^*\|_{0,T} \lesssim h_T^{l_u+2} |u|_{l_u+2,T} + \|\varepsilon^u \|_{0,T} +h_T\, \|\boldsymbol q - \boldsymbol q_h\|_{0,T} + Lh_T^2\|u-u_h\|_{0,T}.
    \end{equation*}
By adding on each $T\in \mc{T}_h$, the estimate \eqref{eq:B1a} is concluded after considering the results in Theorem \ref{thm:estim normH eroor}.  Now, if we apply the inverse inequality to the estimate above, we arrive at
\[
    (1-Lh^2_T)|w|_{1,T} \lesssim \|\boldsymbol q - \boldsymbol q_h\|_{0,T} + L  h_T\,  \ \left( h_T^{l_u+2} |u|_{l_u+2,T} + \|\varepsilon^u \|_{0,T} + \|u-u_h\|_{0,T}\right).
\]   
Assuming again $Lh_T^2<1$ for each $T\in \mc{T}_h$, \eqref{eq:B1b} follows.

Finally, using the trace inequality, the fact that $h_e \|v\|^2_{0,e} \lesssim \|v\|_{0,T} \left(  \|v\|^2_{0,T} + h^2_T |v|^2_{1,T} \right)^{1/2}$ for any $v \in [H^1(K)]^d$, and the estimates \eqref{eq:B1a} and \eqref{eq:B1b}, we have
    \begin{align*}
    \sum_{e\in \mc{E}_h}  h_e \| \jump{u_h^*}\|^2_{0,e} &\lesssim \sum_{e\in \mc{E}_h} \sum_{T'\in \omega_e}  h_e \|u-u_h^*|_{T'}\|_{0,e}^2 \\
    &\lesssim \sum_{e\in \mc{E}_h} \sum_{T'\in \omega_e} \|u-u_h^*\|_{0,T'}\left( \|u-u_h^*\|_{0,T'}^2 + h_{T'}^2 |u-u_h^*|^2_{1,T'} \right)^{1/2},
    \end{align*}
which implies \eqref{eq:B1c}. 
\end{proof}
\section{Auxiliary estimates}
The following results were used throughout the text. We include them here for completeness. 

The first lemma  was needed to bound the terms in the decomposition of $\mb{T}^u$ carried out in Lemma \ref{lem:estim Eu}.

\begin{lem}\label{lem:aux Tu}\cite[Lemma 5.5]{CoQiuSo2014}
Suppose  Assumption \eqref{eq:S4} and the elliptic regularity inequality \eqref{regularity dual problem} hold. Then,
	\begin{subequations}
	\begin{align}
	\|{(h^{\perp})}^{-1/2} (Id_M-P_M) \psi \|_{\Gamma_h} &\lesssim  h \|\Theta\|_{\Omega}, \label{est 1 para T_zh } \\
	\|l^{1/2} (Id_M-P_M)\partial_n \psi \|_{\Gamma_h} &\lesssim R^{1/2} h \|\Theta\|_{\Omega}, \label{est 2 para T_zh } \\
	\|l^{-3/2} (\psi + l \partial_n \psi) \|_{\Gamma_h} &\lesssim \|\Theta\|_{\Omega}, \label{est 3 para T_zh } \\
	\| l^{-1} \psi \|_{\Gamma_h} &\lesssim\|\Theta\|_{\Omega}. \label{est 4 para T_zh }
	\end{align}
	\end{subequations}
\end{lem}
The result below is used when deducing the bound for the term of the estimator involving the jump in the flux.
\begin{lem}\label{lem:ExtensionBound}
Let $e\in \mc{E}^\partial_h$ and $\mbf{v}\in\mbf{H}(\textbf{div};T^e)$. It holds
    \begin{subequations}
    \begin{align}
    \label{eq: estim_operator E a}
    \|E_{T^e}(\mbf{v})\|^2_{T^e_{ext}} \lesssim\,& r_e^2 \, \|\mbf{v}\|^2_{T^e} + r_e^2\, h_T^2 \|\nabla \cdot \mbf{v}\|^2_{T^e} .
    \end{align}
    \end{subequations}
\end{lem}

\begin{proof}
We employ a scaling argument. Let $\Phi: T^e\rightarrow \widehat{T}$ be the affine mapping from $T^e$ to the reference element $\widehat{T}$ and set $\widehat{T_{ext}^e} := \Phi^{-1}(T_{ext}^e)$. We have
    \begin{align*}
\|E_{T^e}(v)\|^2_{T^e_{ext}} &= 2|T^e_{ext}| \| \what{E}(\what{v})\|_{\what{T^e_{ext}}}^2 \lesssim |T^e_{ext}| \| \what{v}\|^2_{H(\textbf{div}; \widehat{T})}  = |T^e_{ext}| \left(  \|\what{v} \|^2_{\what{T}} + \|\what{\nabla} \cdot \what{v} \|^2_{\what{T}} \right) \\
    &\lesssim |T^e_{ext}| \left(  \dfrac{1}{|T^e|} \, \|v\|^2_{T^e} + \|\nabla \cdot v\|^2_{T^e} \right).
    \end{align*}
Thus, considering that $|T^e_{ext}| \lesssim (H_e^{\perp})^2 = R_e^2\, (h_e^{\perp})^2 \leq r_e^2\, h_T^2$, and $|T^e|\lesssim h_T^2$, the inequality \eqref{eq: estim_operator E a} can be deduced. 
\end{proof}

The following result pertaining bubble functions is useful when addressing the local efficiency of the error estimator.

\begin{lem}{\cite[Lemma 3.3.]{verfurth}}\label{lem::bubble_function}
Let $B_T:= \Pi_{i=1}^{d+1} \lambda_i$ be the element--bubble function associated to $T\in \mc{T}_h$, where $\{\lambda_i\}_{i=1}^{d+1}$ are the barycentric coordinates of $T$, and $B_e:= \Pi_{\substack{i=1\\i\neq j}}^{d+1} \lambda_i$ be the face--bubble function associated to $e\subset \partial T$, where $\lambda_j$ vanishes on $e$. Then, the following estimates hold
    \begin{equation}\label{ineq:bubble functions}
        \begin{array}{rlcrlcrl}
         \|\mbf{v}\|^2_T \lesssim& (\mbf{v}, B_T\mbf{v})_T, &\quad& \|B_T\mbf{v}\|_T \lesssim& \|\mbf{v}\|_T, &\quad& \|B_T\mbf{v}\|_{1,T} \lesssim& h_T^{-1}\, \|\mbf{v}\|_T, \\
         \|\mbf{\mu}\|^2_e \lesssim& (\mbf{\mu}, B_e\mbf{\mu})_e, &\quad& \|B_e\mbf{\mu}\|_{\Delta_e} \lesssim& h_e^{1/2}\, \|\mbf{\mu}\|_e, &\quad& \|B_e\mbf{\mu}\|_{1,\Delta_e}, \lesssim& h_e^{-1/2}\, \|\mbf{\mu}\|_e, 
        \end{array}
    \end{equation}
for all $\mbf{v}\in [\md{P}_k(T)]^d, \ T\in \mc{T}_h$ and for each $\mbf{\mu}\in [\md{P}_k(e)]^d, \ e\in \mc{E}_h$.
\end{lem}

\section{Cl\'ement and Oswald interpolants}\label{sec:cle-osw}
%
The following two interpolants are useful in the arguments leading to the reliability of the estimator. They allow to control the behavior of functions with piecewise $H^1$ regularity by representatives belonging to the global $H^1_0(\Omega)$ space. 

First, in the next lemma, we state the approximation properties of the Cl\'ement interpolation operator $\mc{C}_h: L^2(\compD) \to W_h^{1,c} \cap H_0^1(\Omega)$, introduced in \cite{clement} as
\[
    \mc{C}_h w := \sum_{z\in \mc{N}_h}   \left( \dfrac{1}{|\Omega_z|} \int_{\Omega_z} w \ dx \right) \phi_z
\]
where $\phi_z$ is the $\md{P}_1$ nodal basic functions associated to the interior vertex $z$, $\Omega_z = supp \ \phi_z$, and $W_h^{1,c}:= \{ w\in C(\Omega) : w|_T\in \mathbb{P}_1(T), T\in \mc{T}_h \}$.

\begin{lem}{\cite[Lemma 3.2]{verfurth}}\label{lem. Clement interpolator }
For any $T \in \mc{T}_h , \ e\in \mc{E}_h^i$ and $0\leq m\leq 1$, the following estimates hold, for all $w\in H_0^1(\Omega)$
\[
    \| \mc{C}_h w \|_{m,\Omega} \lesssim \|w\|_{m,\Omega}, \quad \|w-\mc{C}_h w\|_{0,T} \lesssim h_T \|w\|_{1,\Delta_T} , \quad \|w-\mc{C}_h w\|_{0,e} \lesssim h_e^{1/2} \|w\|_{1,\Delta_e},
\]
where $\Delta_T:= \{ T'\in \mc{T}_h : \overline{T} \cap \overline{T'} \neq \emptyset \} $ and $\Delta_e = \{T'\in \mc{T}_h: \overline{T'} \cap \overline{e} \neq \emptyset \}$.
\end{lem}

The next results shows that an element $w$ of $W_h^*$ can be approximated by a continuous function $\wtil{w}\in W_h^*$, sometimes referred to as  \textit{Oswald interpolant}, and that the approximation error can be controlled by the size of the inter-element jumps of $w$. 

\begin{lem}{\cite[Theorem 2.2.]{KP}}\label{lem: Oswald Interpolation}
For any $w_h\in W_h^*$ and any multi-index with $|\alpha|=0,1$, the following approximation results holds: Let $u_D$ be the restriction to $\Gamma_h$ of a function in $W_h^*\cap H^1(\compD)$. then there exists a function $\wtil{w}_h\in W_h^*\cap H^1(\compD)$ satisfying $\wtil{w}_h|_{\Gamma} = u_D$, and
\[
	\sum_{T\in \mc{T}_h} \|D^{\alpha} (w_h - \wtil{w}_h)\|^2_T \leq C_O \left(\sum_{e\in \mc{E}_h^\circ} h_e^{1-2|\alpha|} \|\jump{w_h}\|^2_e  + \sum_{e\in \mc{E}_h^{\partial}} h_e^{1-2|\alpha|} \|u_D-w_h\|^2_e\right),
\]		
above, $C_O$ is a positive constant independent of the mesh size.
\end{lem}

\bibliography{biblio}
\bibliographystyle{abbrv}

\end{document}